\documentclass[hidelinks,12pt]{amsart}
\usepackage{tikz,hyperref,stmaryrd,a4wide,amssymb,enumerate}

\title[A dynamical view on the chairman assignment problem]{A dynamical view of Tijdeman's solution of the chairman assignment problem}

 \author[V.~Berth\'e]{Val\'erie Berth\'e}
 \address{Universit\'e Paris Cit\'e, CNRS, IRIF, F-75013 Paris, France}
 \email{berthe@irif.fr}
 \thanks{The first, second, and fourth authors were supported by the Agence Nationale de la Recherche through the project ``SymDynAr'' (ANR-23-CE40-0024). 
  The first and fourth authors were supported by the Agence Nationale de la Recherche through the project ``IZES'' (ANR-22-CE40-0011).  The fifth author was supported by the EPSRC, grant number EP/V007459/2.}
 \author[O.~Carton]{Olivier Carton}
  \address{Universit\'e Paris Cit\'e, CNRS, IRIF, F-75013 Paris, France}
  \email{carton@irif.fr}
 \author[N.~Chevallier]{Nicolas Chevallier}
 \address{Universit\'e de Haute Alsace, IRIMAS d\'epartement de math\'ematiques, 18 rue des frères Lumière, F-68200 Mulhouse, France}
 \email{nicolas.chevallier@uha.fr}
 \author[W.~Steiner]{Wolfgang Steiner}
\address{Universit\'e Paris Cit\'e, CNRS, IRIF, F-75013 Paris, France}
\email{steiner@irif.fr}
 \author[R. Yassawi]{Reem Yassawi}
 \address{School of Mathematical Sciences, Queen Mary University of London, Mile End Road, London, E1 4NS, United Kingdom}
 \email{r.yassawi@qmul.ac.uk}

\thanks{}
\date{\today}

\keywords{Symbolic dynamics; discrepancy; billiard word; toral translation; natural coding}
\subjclass[2010]{}

\newtheorem{lemma}{Lemma}[section]
\newtheorem{theorem}[lemma]{Theorem}

\newtheorem{proposition}[lemma]{Proposition}

\theoremstyle{remark}

\newtheorem{remark}[lemma]{Remark}

\newtheorem{definition}[lemma]{Definition}

\numberwithin{equation}{section}

\newcommand{\balpha}{\boldsymbol{\alpha}}
\newcommand{\be}{\mathbf{e}}
\newcommand{\bp}{\mathbf{p}}
\newcommand{\bq}{\mathbf{q}}
\newcommand{\bx}{\mathbf{x}}
\newcommand{\by}{\mathbf{y}}
\newcommand{\bt}{\mathbf{t}}
\newcommand{\bz}{\mathbf{z}}
\newcommand{\bn}{\mathbf{n}}
\newcommand{\bone}{\mathbf{1}}

\newcommand{\ba}{\mathbf a}

\newcommand{\bu}{\mathbf u}

\newcommand{\N}{\mathbb N}

\newcommand{\Q}{\mathbb Q}
\newcommand{\Z}{\mathbb Z}
\newcommand{\T}{\mathbb T}

\newcommand{\ttt}{\balpha}
\newcommand{\LL}{\Lambda}
\newcommand{\fff}{\rightarrow}
\newcommand{\PP}{\mathcal P}

\newcommand{\KK}{\mathcal K}

\newcommand{\WW}{W}

\newcommand{\dd}{\operatorname{d}}
\newcommand{\card}{\operatorname{card}}
\newcommand{\eps}{\varepsilon}
\newcommand{\pt}{\pi_{\T^d}}

\begin{document} 
\begin{abstract}
In 1980, R.~Tijdeman provided an  on-line algorithm that generates sequences over a finite alphabet  with minimal discrepancy, that is, such that the occurrence of  each letter optimally tracks its frequency.
In this article, we define discrete dynamical systems generating these sequences. The  dynamical systems are defined as exchanges of polytopal pieces, yielding cut and project schemes, and they code tilings of the line whose sets of vertices form model sets.
We prove that these sequences of low discrepancy are natural codings of toral translations with respect to polytopal atoms, and that they generate a minimal and uniquely ergodic subshift with purely discrete spectrum. Finally, we show that the factor complexity of these sequences is of polynomial growth order~$n^{d-1}$, where $d$ is the cardinality of the alphabet. 
\end{abstract}

\maketitle

\section{Introduction}

Quoting  Tijdeman in  \cite{Tijdeman:80}, the chairman assignment problem is stated  as follows:
``Suppose $k$ states form a union and every year a union chairman has to be selected in such a way that at any time the accumulated number of chairmen from each state is proportional to its weight.''  The question is then to give a simple algorithm for a chairman  assignment which guarantees a small discrepancy.
The richness of this problem is that it can be reformulated in several ways, as a sequencing problem in operations research for optimal routing and scheduling, in  terms of word combinatorics, symbolic dynamics and aperiodic order, and also as a discrepancy problem.

In this latter setting, the problem asks for the existence of very well distributed sequences such as  introduced by Niederreiter in \cite{Niederreiter:72bis}. Given a finite alphabet~$\mathcal{A}$ of cardinality $d$ and a vector~$\balpha$ of frequencies for the letters of~$\mathcal{A}$,  the aim is to construct a sequence over~$\mathcal{A}$ in which each letter occurs with its
prescribed frequency as evenly as possible, i.e., where the occurrence of  each letter optimally tracks its frequency; see Section~\ref{sec:notation} for the definition of frequency. 
More precisely, given a sequence $u = (u_k)_{k\in\mathbb{N}} \in \mathcal{A}^{\mathbb{N}}$, 
the \emph{(letter) discrepancy} of~$u$ with respect to~$\balpha$ is defined as 
\[
\Delta_{\balpha}(u) = \max_{i\in\mathcal{A}} \, \sup_{n\in\mathbb{N}} \big|\mbox{Card}  \{0 \le k < n \,:\, u_k=  i\} - n \alpha_i\big|,
\]
for a given vector $\balpha = (\alpha_1, \dots, \alpha_d) \in (0,1)^d$ with $\sum_{i=1}^d \alpha_i = 1$.
It is then natural to  consider, for $d\geq 2$, the quantity 
\[
\textstyle D_d = \sup_{\balpha} \inf _u \Delta_{\balpha}(u),
\]
where the supremum is taken over the set of \emph{frequency vectors} $\balpha$ in the $d$-simplex, and the infimum is taken over the set of sequences~$u$ with values in an alphabet of cardinality~$d$.

The question of finding sequences with minimal discrepancy was raised in \cite{Niederreiter:72,Niederreiter:72bis}.
Niederreiter proved in \cite[Lemma~1]{Niederreiter:72bis} the existence of a sequence  with $\Delta_{\balpha}(u) \leq d{-}1$; see also \cite{MN:72} for refinements. 
He also conjectured that $D_d \leq 1$.
Tijdeman answered this conjecture positively in a nonconstructive fashion in \cite{Tijdeman:73} with a proof based on Hall's marriage theorem, and he also  showed that $D_d \ge 1{-}\frac{1}{2d-2}$.
Refining Tijdeman's proof, Meijer proved in \cite{Meijer:73} that
\[
D_d = 1-\tfrac{1}{2d-2},
\]
but the proof of this result was also nonconstructive.
Lastly, Tijdeman provided in \cite{Tijdeman:80} a linear time on-line algorithm to determine, for each $d$-dimensional vector~$\balpha$,  a sequence~$u$ over a finite alphabet of cardinality~$d$ for which $\Delta_{\balpha}(u)  \leq 1{-}\frac{1}{2d-2}$. We describe Tijdeman's construction in Section~\ref{subsec:tijdeman}. 

These sequences are the object of the present paper.
We call a sequence constructed by Tijdeman's algorithm a \emph{Tijdeman sequence} with frequency~$\balpha$, and we call sequences satisfying $\Delta_{\balpha}(u)  \leq  1{-}\frac{1}{2d-2}$ \emph{fairly distributed}.  
The full definition of Tijdeman sequences (see Definition~\ref{def:tijdeman}) depends on the frequency $\balpha$ and  on three other parameters: two constants $C$ and~$C'$ and a starting point~$\bx_0$. These three parameters can be used to optimize the discrepancy.

As recalled in \cite{Tijdeman:survey}, the algorithm proposed by Tijdeman  in \cite{Tijdeman:80} for the chairman assignment problem is closely  related to the quota-method of Balinski and Young \cite{BY:75,BY:77,BY:85} for the (discrete)  apportionment problem; see \cite{CembranoCorreaVerdugo2022} for more references on the subject. 
This problem, which has its origins in the problem of seat assignments to the house of representatives in the United States, consists in allocating seats in a proportional way. 
See also \cite{Li:22} for the connection with apportionment problems and Just-In-Time sequencing problems (with maximal deviation JIT scheduling), such as considered in  
\cite{AltmanGH:00,BraunerCrama:04,BraunerJost:08}, and see the survey \cite{Vuillon:03}
for more references. 
See in addition  \cite{Coppersmith:11}, where it is proved that the greedy algorithm is optimal among online algorithms for the chairman assignment problem.
 
As mentioned  above, the richness of this problem is that it goes well beyond  the scheduling framework: indeed, it can be reformulated combinatorially in terms of balance in word combinatorics, dynamically in terms of symbolic codings of toral translations, arithmetically in terms of bounded remained sets or else  in terms of cut and project schemes, within the setting of aperiodic order; the latter  corresponds to  the mathematical formalization of quasicrystals. See for instance \cite{BG13} for a general and in-depth insight into aperiodic order.
Let us discuss now these complementary and intricately connected notions.
 
In combinatorial terms, discrepancy is closely related to  the notion of balance. 
Given a finite alphabet~$\mathcal{A}$, a sequence $u \in \mathcal{A}^{\mathbb{N}}$ is said to be \emph{$B$-balanced} if there exists a constant~$B$ such that for every letter $i \in \mathcal{A}$ and for every pair $(w,w')$ of factors of~$u$ of the same length, the difference between the number of occurrences of~$i$ in $w$ and $w'$ differs by at most~$B$.  
Balance was first studied in the form of $1$-balance for binary sequences by Morse and Hedlund in the seminal papers \cite{MorseHedlund:38,MorseHedlund:40}, in which they  laid the basis for symbolic dynamics  (see Section~\ref{sec:symbolic-dynamics}): 
the binary $1$-balanced aperiodic sequences are exactly the \emph{Sturmian sequences}. 
The notion of balance was then considered for larger alphabets, for $B$-balance, with $B>1$, and for factors instead of letters. 
For more on the subject, see the survey \cite{Vuillon:03}.  
Words over a larger alphabet that  are $1$-balanced have been characterized in \cite{Hubert:00} (see also \cite{Graham:73}) and shown to be closely related to Sturmian words.  
Let us observe that their letter frequencies are rationally dependent: there exists $\gamma$ such that $\alpha_i \in \mathbb{Z} {+} \gamma \mathbb{Z}$, for $1 \leq i \leq d$, by \cite[Lemma~4.1]{Hubert:00}. 
Note that, in ergodic terms, balance can be interpreted as an optimal speed of convergence of Birkhoff sums toward frequencies of words. 
A~sequence is $B$-balanced if and only if it has  finite discrepancy (see \cite{Adam:03,Adam:04}), and a fairly distributed sequence over two letters must be $1$-balanced; see Proposition~\ref{prop:sturm}.

There is a classical way of constructing balanced sequences, and in particular Sturmian sequences, in terms of \emph{cutting sequences}. 
Specifically, Sturmian sequences, which are defined on binary alphabets, are codings of trajectories of billiards on a square table with an irrational direction and, by unfolding trajectories, they code whether a horizontal or a vertical side of a lattice square is hit.
They have been generalized to larger alphabets as hypercubic billiard sequences, as presented in \cite{Ar.Ma.Sh.Ta.94,Ar.Ma.Sh.Ta.94bis}; see also Section~\ref{subsec:billiard}. 
An equivalent   description is in terms of natural codings of toral translations; see Definition~\ref{sec:rotation}. 
Indeed, Sturmian sequences are known to be symbolic codings of a special kind, namely a Sturmian sequence codes an exchange of two intervals, which happens to be a translation of the one-dimensional torus $\mathbb{R}/\mathbb{Z}$.

We focus here on the case where the vector of frequencies~$\balpha$ has linearly independent entries over the rationals. Under this hypothesis, a~lower bound for the discrepancy is given in \cite[Theorem~2]{Schneider:96} by $\Delta_{\balpha}(u) \geq 1 {-} \frac{1}{d}$ for all sequences~$u$.
When $d= 2$, this gives $\Delta_{\balpha}(u) \geq \frac{1}{2} = D_2$.
This bound is achieved by sequences defined in terms either of Beatty sequences, or, in some equivalent way, of Sturmian sequences (see in particular Remark~\ref{rem:sturm} below): for $\balpha = (\alpha,1{-}\alpha)$,  the sequence $(u_k)_{k\in \mathbb{N}}$, defined by $u_k=1$ if and only if  $k = \lceil(n{-}\frac{1}{2})/\alpha\rceil$ for some nonnegative $n$, satisfies $\Delta_{\balpha}(u) = \frac{1}{2}$, with the sequence $(\lceil(n{-}\frac{1}{2})/\alpha\rceil)_n$  being  known as a \emph{Beatty sequence}. 
We recall that the first difference sequence of a Beatty sequence is a Sturmian sequence; and that Sturmian sequences are $1$-balanced,  as discussed above.
  
The aim of this paper is to provide a similar dynamical description for Tidjeman sequences, which, as mentioned above, are among the sequences having the lowest letter discrepancy $\Delta_{\balpha}(u)$ with $\Delta_{\balpha}(u) \leq  D_d = 1 {-} \frac{1}{2d-2}$, for vectors of frequencies~$\balpha$ with rationally independent coordinates. 
We present them as \emph{natural codings} of a dynamical system defined as a higher dimensional domain exchange with polytopal pieces, which turns out to be a translation of the torus; see Definition~\ref{def:coding}.

Let  $\balpha $ be a frequency vector, i.e., a vector in $(0,1)^d$ with $\sum_{i=1}^d  \alpha_i=1$. We say that $\balpha $ is {\em totally irrational} if  it has rationally independent coordinates.  This is equivalent to demanding that the coordinates of  $ (\alpha_1,\dots,\alpha_{d-1},1)$ are rationally independent. 
We consider the (minimal\footnote{See Section~\ref{sec:symbolic-dynamics} for the definition of minimality.}) toral translation by $(\alpha_1,\dots,\alpha_{d-1})$, denoted as~$T_{\balpha}$, defined on $\mathbb{T}^{d-1} = \mathbb{R}^{d-1}/\mathbb{Z}^{d-1}$ by\footnote{Observe that the  translation   $T_{\balpha}$ is defined   on    $\mathbb{T}^{d-1}$,  and not on 
 $\mathbb{T}^{d}$; this is due to the fact that $\balpha$ being assumed to be a frequency  vector,   $\sum_{i=1} ^d \alpha_i=1$.}
\[
T_{\balpha}:\  \mathbb{T}^{d-1} \rightarrow \mathbb{T}^{d-1}, \quad  \bx \mapsto \bx + (\alpha_1,\dots,\alpha_{d-1})  \pmod{\mathbb{Z}^{d-1}}.
\]
The codings we consider are called natural, by which we mean that these sequences code translations  $T_{\balpha}$ of the torus with respect to (polytopal) partitions of~$\mathbb{T}^{d-1}$ that come from domain exchanges in~$\mathbb{R}^{d-1}$ with translation vectors equal to $(\alpha_1,\dots,\alpha_{d-1}) \bmod \mathbb{Z}^{d-1}$; see Figure~\ref{f:finiteunionsbis} and also Definition~\ref{def:coding}.
These codings are even \emph{bounded} natural codings, in the sense that the pieces are bounded as subsets of~$\mathbb{R}^{d-1}$.
This dynamical description allows us to  deduce estimates on the factor complexity of Tijdeman sequences, in particular proving that the factor complexity is of order~$n^{d-1}$, when defined over an alphabet of cardinality~$d$. 

Our main result is 

\begin{theorem}\label{theo:main}
Let $\balpha = (\alpha_1, \dots, \alpha_d)$ be a totally irrational frequency vector. Then there exist Tijdeman parameters generating a sequence~$u$ with $\Delta_{\balpha}(u) \leq 1{-}\tfrac{1}{2d-2}$, and such that $u$ is the bounded natural coding of~$T_{\balpha}$, via a partition of a fundamental domain of $\mathbb{R}^{d-1} / \mathbb{Z}^{d-1}$ into $d$ finite unions of convex polytopes.\footnote{We follow the usual convention that a convex polytope is the convex hull of a finite number of points. Hence here the fundamental domains under consideration are bounded.}
Furthermore, the shift defined by~$u$ is minimal, uniquely ergodic, has purely discrete spectrum and factor complexity of order~$n^{d-1}$.
\end{theorem}

Theorem~\ref{theo:main} induces the existence of  a ``good''  symbolic coding for any minimal toral translation~$T_{\balpha}$ such that the $d{-}1$ first coordinates of $\balpha$ are in the polyhedron $\mathcal{T}=\{(\alpha_1,\ldots, \alpha_{d-1}) : \alpha_i\geq 0 \ \text{and}  \,  \sum_{i=1}^{d-1}\alpha_i\leq 1\}$.  
By changing the signs of the coordinates of $(\alpha_1,\ldots, \alpha_{d-1})$,  we  also obtain good codings for all~$\balpha$ with $\sum_{i=1}^{d-1}|\alpha_i|\leq 1$. 
Hence, when $d{-}1=2$, we obtain a good coding for any minimal toral translation. However, when $d>3$, we do not obtain all the toral translations. 
Neverthless, thanks to the following  standard argument (see for example \cite[Remark~3.4]{BST:20} or \cite[Section~10.2]{Fogg:20}), it is possible to find a good coding for every minimal translation even if $d>3$. Indeed, the cube $[0,1]^{d-1}$ is a union of $(d{-}1)!$ images of the polyhedron~$\mathcal{T}$ by some transformations in  $\operatorname{GL}_{d-1}(\Z)$, so that any translation $T_{\boldsymbol{\beta}}: \T^{d-1} \fff \T^{d-1}$ can be written as $T_{\boldsymbol{\beta}} = gT_{\balpha}g^{-1}$ for some $\balpha \in \mathcal{T}$ and $g\in\operatorname{GL}_{d-1}(\Z)$. This implies that  the codings associated with~$T_{\balpha}$  with respect to some  partition of the torus share the same properties as the codings associated with~$T_{\boldsymbol{\beta}}$ and the image by~$g$ of this partition.

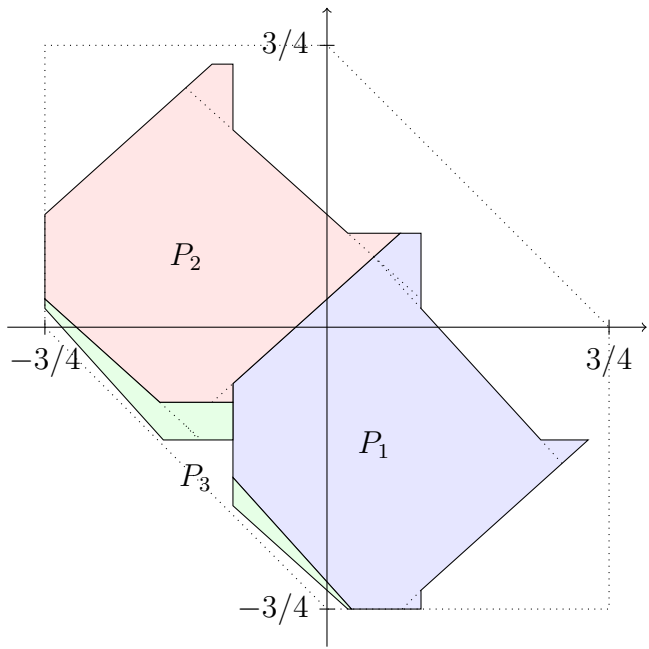
\begin{figure}[ht] 
\begin{tikzpicture}[scale=5]
\newcommand{\alphaone}{.5}
\newcommand{\alphatwo}{.45}
\newcommand{\alphathree}{.05}
\newcommand{\Pone}{\filldraw[fill=blue!10](-.25,-.15)--(-.25,-.4)--(.065,-.75)--(.25,-.75)--(.25,-.7)--(.69444,-.3)--(.56875,-.3)--(.25,.05)--(.25,.25)--(.19444,.25)--cycle;
\draw[dotted](.2,-.75)--(.25,-.7) (.25,.05)--(.125,.1875)--(.25,.075) 
(.625,-.3625)--(.56875,-.3);} 
\newcommand{\Ptwo}{\filldraw[fill=red!10](-.44444,-.2)--(-.25,-.2)--(-.25,-.15)--(.19444,.25)--(.05556,.25)--(-.25,.525)--(-.25,.7)--(-.30556,.7)--(-.75,.3)--(-.75,.075)--cycle; \draw[dotted](-.30556,-.2)--(-.25,-.15) (.125,.1875)--(.05556,.25) (-.375,.6375)--(-.25,.525);}
\newcommand{\Pthree}{\filldraw[fill=green!10](-.44444,-.2)--(-.25,-.2)--(-.25,-.3)--(-.435,-.3)--(-.75,.05)--(-.75,.075)--cycle (-.25,-.4)--(.065,-.75)--(.05556,-.75)--(-.25,-.475)--cycle; \draw[dotted](-.375,-.2625)--(-.34125,-.3) (-.44444,-.2)--(-.33333,-.3);} 
\newcommand{\coord}{\draw[->](-.85,0)--(.85,0); \draw[->](0,-.85)--(0,.85); \draw(-.75,-.02)node[below]{${-}3/4$}--(-.75,.02) (.75,-.02)node[below]{$3/4$}--(.75,.02) (-.02,-.75)node[left]{${-}3/4$}--(.02,-.75) (-.02,.75)node[left]{$3/4$}--(.02,.75); \draw[dotted](-.75,0)--(0,-.75)--(.75,-.75)--(.75,0)--(0,.75)--(-.75,.75)--cycle;}
\Pone \node at (.125,-.3125){$P_1$};
\Ptwo \node at (-.375,.1875){$P_2$};
\Pthree \node at (-.35,-.4){$P_3$};
\coord
\end{tikzpicture}
\caption{A fundamental domain of $\mathbb{R}^2 / \mathbb{Z}^2$ and its partition by finite unions of polygons such that the natural codings of the action of $T_{\balpha}$ are Tijdeman sequences with $\balpha \approx (0.5,0.45,0.05)$, $C = C' = D_3 = 3/4$.} \label{f:finiteunionsbis} 
\end{figure}
  
  The fact that the  factor complexity is bounded below by $Cn^{d-1}$ implies that    a Tijdeman sequence  cannot generate a primitive substitution shift for $d\geq 3$, since points in primitive substitution shifts have linear factor complexity; see e.g.\ \cite{Queffelec:2010}.  In addition, since Tijdeman sequences are
   bounded natural codings of totally irrational  rotations, 
   their shifts cannot have group factors  which are  periodic, and in particular they cannot be Toeplitz sequences; see e.g. [Dow05].
     Note also that similar results have been stated in terms of estimation of balance in order to get sequences having low balance.
See for example \cite{LMP:23} for a construction of 2-balanced sequences over a three-letter and a four-letter alphabet based on Sturmian sequences; see also \cite{BCS}.
  
Let us explain why the toral translation~$T_{\balpha}$ intervenes naturally in the present setting. In order  to produce a sequence~$u$ with a small discrepancy, one constructs, step-by-step, a  half broken line~$\mathbf{L}_u$ (a~half discrete line) whose vertices belong to $\mathbb{N}^d$ and that ``optimally'' approximates (in a sense to be defined) the half line~$\mathbb{R}^+\balpha$.
Given a $d$-letter alphabet~$\mathcal{A}$, we associate with each distinct letter of~$\mathcal{A}$ a vector of the canonical basis for~$\mathbb{R}^d$. Using this association, each sequence $u \in \mathcal{A}^{\mathbb{N}}$ defines a half broken line~$\mathbf{L}_u$ that lives in~$\mathbb{N}^d$; see Figure~\ref{fig:brokenline} for an illustration.
To define a sequence that does not stray from~$\mathbb{R}^+\balpha$,  we must establish a strategy that allows us to optimally choose  the value of~$u_n$ when $u_0 \cdots u_{n-1}$ are determined. Geometrically, the discrepancy~$\Delta_{\balpha}(u)$  measures how far the vertices of the broken line are from the half line~$\mathbb{R}^+\balpha$. Consider indeed the projection along the direction~$\mathbb{R}^{+}\balpha$ onto some given transverse hyperplane that does not contain~$\balpha$.
The discrepancy is calculated by measuring the distance between the projection of the set of vertices of the broken line and the origin.  For suitable choices of sequences~$u$, the closure of the projection of the set of  vertices of the broken line
form (after some natural change of variables) a fundamental domain for the action of  the lattice~$\mathbb{Z}^{d-1}$ on~$\mathbb{R}^{d-1}$. Moreover, moving on the broken line  by one step, i.e., by some canonical vector, consists in  moving in the fundamental domain by~$T_{\balpha}$.  Thus the sequence~$u$ codes the action of~$T_{\balpha}$ with respect to a  finite partition.
The atom~$P_a$ of the partition of the fundamental domain is obtained by taking the closure of the projection of the vertices of the broken line when the following letter equals~$a$; see Figure~\ref{f:finiteunionsbis} as an illustration  of the atoms of the partition. Our approach  has to be compared to  the one developed  in \cite{Chevallier}  which holds in dimension $d=3$ and where the choice of the vertices of the broken line is done with respect to the Euclidean norm (whereas here we rely on the supremum norm). 

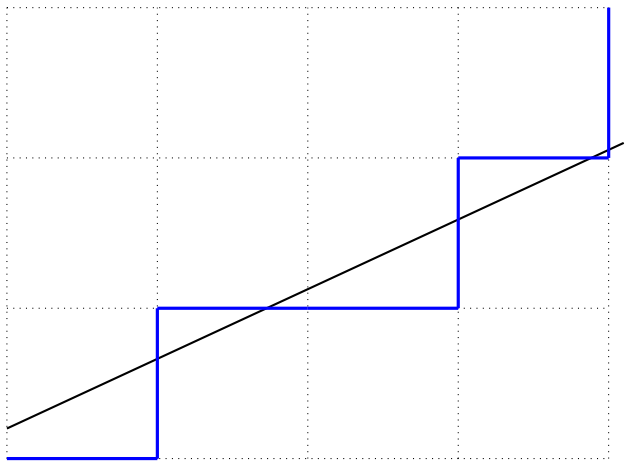
\begin{figure}
\centerline{\begin{tikzpicture}[scale=2]
\draw[dotted](0,0) grid (4,3);
\draw[thick](0,.2)--(4.1,2.1);
\draw[blue,very thick](0,0)--(1,0);
\draw[blue,very thick](1,0)--(1,1);
\draw[blue,very thick](1,1)--(2,1);
\draw[blue,very thick](2,1)--(3,1);
\draw[blue,very thick](3,1)--(3,2);
\draw[blue,very thick](3,2)--(4,2);
\draw[blue,very thick](4,2)--(4,3);
\end{tikzpicture}}
\caption{A piece of a  broken line associated with the sequence $12 11 212\cdots$.} \label{fig:brokenline}
\end{figure}

Bounded remainder sets also play a central role here. 
A~\emph{bounded remainder set} for  the translation~$T_{\balpha}$ acting on the $(d{-}1)$-dimensional torus~$\mathbb{T}^{d-1}$ is a measurable subset~$A$ of~$\mathbb{T}^{d-1}$ for which there exists $C>0$ such that, for any point $\bx \in \mathbb{T}^{d-1}$ and any $n \in \mathbb{N}$, 
\[
|\mbox{Card}  \{0\leq k < n \,:\,  T^k_{\balpha} (\bx)\in A \} - n\mu(A) | \leq C.
\]
These are sets having bounded local discrepancy, i.e., the difference between the number of visits to this particular set and its expected value is bounded. 
Their study started with the work of Schmidt in his series of papers on  irregularities of distributions initiated in \cite{Schmidt:68}.
Grepstad and Lev have given in \cite{GL:15} a particularly nice family of   bounded remainder sets for  the minimal  translation  $T_{\balpha}$ as the parallelotopes in $\mathbb{R}^{d-1}$ spanned by vectors belonging to $\mathbb{Z} (\alpha_1,\dots,\alpha_{d-1}) + \mathbb{Z}^{d-1}$. This family generalizes Kesten's characterization of the intervals that are bounded remainder sets of the  unit circle for an  irrational translation  by~$\alpha$ as the intervals of length in $\alpha \mathbb{Z} + \mathbb{Z}$. 
The atoms of the partition from Theorem~\ref{theo:main} are bounded remainder sets.

Lastly, let us briefly reinterpret the previous notions in terms of cut and project schemes and model sets. Model sets play a prominent role as mathematical models for quasicrystals. They have been introduced by Meyer in \cite{Meyer:72}. For more details, 
see for instance \cite{Moody:00,BG13}, and  the references in  \cite{KW:19}. Note also   that using the cut and project schemes, Kesten's theorem has been  reformulated in the language of aperiodic tiling in \cite{KS15} with a proof relying on tiling cohomology.
We recall  here  the corresponding definition in the simple Euclidean setting of the present paper. 

We start by defining a \emph{cut and project scheme}. We consider the full rank lattice $\mathbb{Z}^d$ in~$\mathbb{R}^d$, together with  the decomposition of~$\mathbb{R}^d$ as the direct sum of two subspaces, namely  a line~$\mathbf{L}$ directed by~$\balpha$ (called the \emph{physical space})  and the hyperplane~$\balpha^\perp$ 
orthogonal to the  line~$\mathbf{L}$  (called the \emph{internal space}).
The lattice~$\mathbb{Z}^d$ will be projected  orthogonally onto  the line~$\mathbf{L}$ (the inner space~$\balpha^\perp$ determines the direction of the projection map). 
Let $\tilde{\pi}$ denote the corresponding projection. 
We also consider the projection~$\pi$ onto~$\balpha^\perp$ along~$\mathbf{L}$.
We further assume that the restriction of~$\tilde{\pi}$ to~$\mathbb{Z}^d$ is injective and that ${\pi}(\Z^d)$ is dense in~$\balpha^{\perp}$.
This holds true if $\balpha$ is a totally irrational frequency vector, by Kronecker's theorem. 
Model sets are then formed by projections together with a way of selecting points (the cutting part), and this selection is done thanks to an \emph{acceptance window}~$W$ that lives in the  internal space~$\balpha^\perp$. 
Then, a~subset~$\Lambda$ of~$\mathbb{R}^d$ is a \emph{model set}  (associated with  the cut and project scheme developed above) if there exists a precompact set~$W$ of the internal space~$\balpha^\perp$ with nonempty interior\footnote{Note that interiors are considered with respect to the ambiant space $\balpha^\perp$.}  such that 
\[
\Lambda = \{\tilde{\pi}(\bx) \,:\,  \bx  \in \mathbb{Z}^d ,  \  {\pi} (\bx) \in  W \}.
\]
It proves to be more convenient in the present work to consider a window~$W$ in the hyperplane~$\bone^\perp$ of vectors whose sum of coordinates equals~$0$, together with the projection~$\pi_{\balpha}$ onto~$\bone^\perp$ along~$\mathbf{L}$ (instead of the projection~$\pi$).
We choose suitable acceptance windows (as finite unions of convex polytopes) such that the set of points $\{\bx \in \mathbb{N}^d \,:\, \pi_{\balpha}(\bx) \in  W\}$ forms a half broken line~$\mathbf{L}_u$ associated with some sequence~$u$, with also the sequence~$u$ having small discrepancy. The projection by~$\tilde{\pi}$ of the set of vertices of~$\mathbf{L}_u$ is then the positive part of a model set. Moreover, since the model sets are \emph{regular}, i.e., the boundary of the acceptance window has zero Lebesgue measure, and because these sequences~$u$ generate minimal uniquely ergodic dynamical systems with discrete spectrum, every dynamical eigenvalue has a continuous eigenfunction; see \cite{BLM:2007} and references therein. Thus sequences in these systems are Weyl almost periodic; see \cite[Corollary~5.6]{LSS:2024} for details.

\smallskip
We now  sketch the contents of this paper.  
In Section~\ref{sec:notation}, we recall basic definitions, in particular basics from symbolic dynamics, and the notions of exchange of domains and that of a natural coding. In Section~\ref{sec:cons}, we recall and compare two constructions of sequences with small discrepancy, hypercubic billiard sequences in Section~\ref{subsec:billiard}, and then Tijdeman's construction from \cite{Tijdeman:80} in Section~\ref{subsec:tijdeman}.  
We prove that both constructions are obtained by coding the same toral translation~$T_{\balpha}$ with respect to partitions by  finite unions of  convex polytopes.
The proof of Theorem~\ref{theo:main} is given in Section~\ref{sec:proof}.
We end  this paper with questions in Section~\ref{sec:comments}.

\subsection*{Acknowledgements}
The authors  thank J.~Abou Samra for his careful reading and J.~Walton for his insight   on factor complexity for cut and project sets. The authors also thank the referees for their excellent suggestions.

\section{Notation and basic definitions}\label{sec:notation}

\subsection{Word combinatorics and symbolic dynamics} \label{sec:symbolic-dynamics}

Let $\mathcal{A} = \{1,2,\dots,d\}$ be a finite alphabet.  
We denote by~$\varepsilon$ the empty word of the free monoid~$\mathcal{A}^*$, and by $\mathcal{A}^{\mathbb{N}}$ the set of sequences over~$\mathcal{A}$. 
For $i \in \mathcal{A}$ and for $w \in \mathcal{A}^*$, let $|w|_i$ denote the number of occurrences of the letter~$i$ in the word~$w$, and let $|w|$ denote the length of~$w$. 
The $k$-th letter of~$w$ is denoted as~$w_k$, where we always label indices starting at~$0$, i.e., $w=w_0 w_1 \cdots w_{|w|-1}$. 
If $w$ is a word or a sequence, and $j \leq k$ are nonnegative integers, let $w_{[j,k)} := w_j \cdots w_{k-1}$ (where $w_{[k,k)}$ is the empty word). 
In a finite word or sequence~$u$, any word of the form $u_{[j,k)}$ is called a \emph{factor}.
The set of factors~$\mathcal{L}(u)$ of a sequence~$u$ is called its
\emph{language}. 
The \emph{factor complexity} of the sequence~$u$ is the function which, given $n \in \mathbb{N}$, counts the number of factors of~$u$ of length~$n$. 
Given a word~$w$ and a letter~$i$, let $|w|_i$ denote the number of occurrences of~$i$ in~$w$, and for $w \in \mathcal{A}^*$, let 
\[
\bp(w) = (|w|_i)_{i\in\mathcal{A}}
\]
denote its \emph{Parikh vector}.

\subsubsection*{Shifts} 
Let $S$ denote the \emph{shift map}  acting on~$\mathcal{A}^{\mathbb{N}}$, i.e., 
$S((u_n)_{n \in \mathbb{N}}) = (u_{n+1})_{n \in \mathbb{N}}$. 
A~\emph{shift} is a pair $(X,S)$ where $X$ is a closed shift-invariant subset of some~$\mathcal{A}^{\mathbb{N}}$; $X$~is called a \emph{shift space}. 
Here, $\mathcal{A}$~is equipped with the discrete topology, and $\mathcal{A}^{\mathbb{N}}$ is equipped with the product topology.
One associates with any sequence $u \in \mathcal{A}^{\mathbb{N}}$ the  symbolic dynamical system $(X_u,S)$, where the shift space $X_u \subset \mathcal{A}^{\mathbb{N}}$  is defined as $X_u = \{v\in \mathcal{A}^{\mathbb{N}} : \mathcal{L}(v) \subset \mathcal{L}(u)\}$.
A~shift $(X,S)$ is said to be \emph{minimal} if $X$ admits no nontrivial closed and shift-invariant subset. 
If $X$ is a shift space, then its \emph{language}~$\mathcal{L}(X)$ is defined as the set of factors of elements of~$X$.
For any $n\geq 1$, we let $\mathcal{L}_n(X)$ denote the set of factors of length~$n$ of elements in~$X$.
The \emph{factor complexity}~$p_X(n)$ of a minimal shift $(X,S)$ is defined  as 
$p_X(n) = \mathrm{Card}\, \mathcal{L}_n(X)$.

\subsubsection*{Frequencies and invariant measures} 
Let $u$ be a sequence in~$\mathcal{A}^{\mathbb{N}}$ and let $v \in \mathcal{A}^*$. 
If the limit $\lim_{n \to +\infty} \big(|u_{[0,n)}|_v\big)/n$ exists, then we call it the \emph{frequency} of~$v$ in~$u$, and denote it by~$\alpha_v$ (suppressing the dependence on~$u$).
Assume that $u$ is such that the frequencies of the factors of~$u$ all exist.
Then $u$ is said to have \emph{uniform frequencies} if, for every word~$v$, the convergence $|u_{[k,k+n)}|_v/n \to \alpha_v$ is uniform in~$k$. 

Let $(X,S)$ be a shift with $X \subset \mathcal{A}^{\mathbb{N}}$. 
A~probability measure~$\mu$ on~$X$ is said to be \emph{$S$-invariant} if $\mu(S^{-1} B) = \mu(B)$ for every Borel set $B \subset X$.
The shift $(X,S)$ is \emph{uniquely ergodic} if there exists a unique shift-invariant probability measure on~$X$; this is the case if and only if every word $u \in X$ has uniform factor frequencies \cite[Proposition~7.2.10]{CANT}. 
In that case, one recovers the frequency~$\alpha_v$ of a factor $v = v_0 \cdots v_{n-1}$   as $\alpha_v = \mu([v])$, with the \emph{cylinder} $[v] : = \{u \in X \,:\, u_{[0,n)} = v\}$.
For more  on invariant measures and ergodicity, we refer to \cite{Queffelec:2010} and \cite[Chap.~7]{CANT}.

\subsubsection*{Balance and discrepancy}
Let $u \in \mathcal{A}^{\mathbb{N}}$.
A~sequence~$u$ is said to be $B$-\emph{balanced} if there exists a constant~$B$ such that, for every letter $i \in \mathcal{A}$ and for every pair $(w,w')$ of words in~$\mathcal{L}(u)$ with $|w| = |w'|$, we have $|w|_i - |w'|_i \leq B$.
We then define
\[
B_u =   \max_{i\in\mathcal{A}}  \sup_{w, w' \in \mathcal{L}(u) \,:\, |w|=|w'|} \big(|w|_i-|w'|_i\big).
\]
Note that  the first papers devoted to balance were concerned with $1$-balance for letters; see e.g.\ \cite{Lothaire:2002}.  
Balanced sequences $u$ define shifts that are {\em plastic}: all tiling spaces defined by $u$, by varying the vector of lengths associated to letters, are topologically conjugate \cite{S:2016}.

Let $u \in \mathcal{A}^{\mathbb{N}}$ and assume that letters~$i$ admit frequencies~$\alpha_i$ in~$u$. 
The \emph{(letter) discrepancy} of~$u$ is defined as the (possibly infinite) quantity
\[
\Delta_{\balpha}(u) = \sup_{n\in\mathbb{N}} \|\bp(u_{[0,n)})  - n \balpha\|_\infty,
\]
with the frequency vector $\balpha = (\alpha_i)_{i\in\mathcal{A}}$, and $n \balpha {-} \bp(u_{[0,n)}) $ is called a \emph{discrepancy vector}.

These definitions extend to any  minimal shift $(X,S)$ in a  straightforward way.

Let $u$ be a sequence for which the letter frequencies exist. 
It has bounded letter balance if and only if the discrepancy is finite; see \cite{Adam:03,Adam:04}. 
Moreover, one has 
\[
\Delta_{\balpha}(u) \leq B_u \leq 4 \Delta_{\balpha}(u)
\]
by the triangle inequality; see \cite[Proposition~7 and Remark~8]{Adam:03}.  
One advantage of the notion of balance is that one does not need to know in advance the frequency vector~$\balpha$.
Balance is also equivalently formulated as having bounded  abelian complexity, where the abelian complexity counts the number of distinct Parikh vectors of factors of a  given length; see e.g.\ \cite{RSZ}. For related results on the abelian complexity, see \cite{LMP:23}.

The following holds for  fairly  distributed sequences over small alphabets for rationally independent  coordinates. By  Remark \ref{rem:sturm}, not all Sturmian sequences
are fairly distributed.

\begin{proposition} \label{prop:sturm} 
Let $\balpha$ be a totally irrational frequency vector, and let $u$ be a fairly distributed sequence  with letter frequency~$\balpha$, i.e., $\Delta_{\balpha} (u) \leq 1 {-} \frac {1}{2d-2}$.
If $d=2$, then $u$ is Sturmian and its balance~$B_u$ equals~$1$.
If $d=3$, then $B_u \leq 2$.
\end{proposition} 

\begin{proof}
Consider two factors~$w$ and~$w'$ of~$u$ of the same length and a letter~$a$ in the alphabet~$\mathcal{A}$ on which $u$ is defined. Let $n$ be such that both factors $w$ and~$w'$ occur in~$u$ at indices smaller than~$n$.
Since $\alpha$ is irrational, one has, for $0 \leq k \leq n$, $|u_{[0,k)}|_a {-} k  \alpha_a$ irrational, hence the value $1 {-} \frac {1}{2d-2}$ is not attained and $|u_{[0,k)}|_a {-} k \alpha_a < 1 {-} \frac {1}{2d-2}$ for $0\leq k\leq n$.
By the triangle inequality, 
\[
\big||w|_a{-} |w'|_a\big| \leq 4 \max_{0  \leq k \leq n} \big| |u_{[0,k)}|_a {-} k  \alpha_a \big| <  4\,  \big(1 {-} \tfrac {1}{2d-2}\big).
\]
Assume $d=2$. Since  $|w|_a{-} |w'|_a$ takes integer values, $|w|_a{-} |w'|_a \leq 1$, and so its balance~$B_u$ satisfies $B_u \leq 1$.
Since $u$ is one-sided  and $\balpha$ is  totally irrational, then $B_u=1$ and $u$ is a Sturmian sequence.
If $d=3$, we similarly get that $B_u \leq 2$.
\end{proof}

\subsubsection*{Discrete spectrum}
Recall that two measure preserving systems $(X,S,\mu)$ and $(Y,T,\nu)$ are \emph{measurably conjugate} if there are measurable sets $X_0\subset X$ and
$Y_0\subset Y$, each of measure~$1$, and a measurable bijection $\Phi: X_0 \to Y_0$ which intertwines the action, $T \circ \Phi= \Phi \circ S$.
A~shift $(X,S,\mu)$ has \emph{purely discrete spectrum} if the measurable eigenfunctions of the \emph{Koopman operator} $U_S: L^2(X,S,\mu) \to L^2(X,S,\mu)$, $f \mapsto f \circ S$, span $L^2(X,S,\mu)$.
This is equivalent to the shift being  measurably conjugate to a translation on a compact abelian group. 
If the system $(X,S)$ is uniquely ergodic, then we write $(X,S)$ instead of $(X,S,\mu)$. 
Here, the shifts under consideration will be conjugate to $(\mathbb{T}^{d-1}, T_{\balpha})$.

\subsection{Exchange of pieces, toral translations and natural codings} \label{sec:rotation}
Let $\balpha$ be a frequency vector, i.e., a~vector in $(0,1)^d$ with $\sum_{i=1}^d  \alpha_i=1$. We recall that $\balpha$ is said to be totally irrational if it has rationally independent coordinates.
We consider a frequency vector $\balpha = (\alpha_1,\dots,\alpha_d) \in (0,1)^d$ and the toral translation by~$\balpha$, denoted as~$T_{\balpha}$, defined on $\mathbb{T}^{d-1} = \mathbb{R}^{d-1}/\mathbb{Z}^{d-1}$ by
\begin{equation} \label{eq:talpha}
T_{\balpha}:\,  \mathbb{T}^{d-1} \rightarrow \mathbb{T}^{d-1}, \quad  \bx \mapsto \bx + (\alpha_1, \dots, \alpha_{d-1}) \pmod{\mathbb{Z}^{d-1}}.
\end{equation}
A~\emph{topological dynamical system}, i.e., a pair $(X,T)$ where $X$ is a compact metric space and $T: X \to X$ is a homeomorphism, is \emph{minimal} when it does not contain any non-empty proper closed $T$-invariant subset. 
In other words, minimality for $(\mathbb{T}^{d-1}, T_{\balpha})$ means that the orbit of  any point is dense in~$\mathbb{T}^{d-1}$; it is equivalent to the fact that $\balpha$ is totally irrational (i.e., the coordinates of $ (\alpha_1,\dots,\alpha_{d-1},1)$ are rationally independent).
We recall that a minimal translation is also uniquely ergodic; see e.g.\ \cite{Walters:82}. 

We want to provide symbolic codings of the  translation~$T_{\balpha}$ with respect to finite partitions (by  polytopes) of fundamental domains of $\mathbb{R}^{d-1} / \mathbb{Z}^{d-1}$. 
We consider in particular partitions of fundamental domains of $\mathbb{T}^{d-1}$  that are well adapted to the action of~$T_{\balpha}$, in the sense that \emph{on each atom the map $T_{\balpha}$ is a translation by a vector}. 
They are called \emph{natural partitions}.  
Let us state a few definitions in order to make this notion more precise. For the next definition, we follow \cite[Section~2.4]{BST:20}, originally stated in the case that $V = \mathbb{R}^{d-1}$ and $\Lambda = \mathbb{Z}^{d-1}$. 
We  denote by $\operatorname{Leb}_V$, or $\operatorname{Leb}$ when there is no ambiguity, the Lebesgue measure on a finite dimensional real vector space~$V$.
The interior of a set~$P$ is denoted as~$\mathring{P}$, its closure as~$\overline{P}$,  and its boundary as~$\partial{P}$.

\begin{definition}[Fundamental domains and natural partitions]\label{def:coding} 
Let $V$ be a finite dimensional  real vector space, $\LL$~a full rank lattice in~$V$, and   $\balpha \in  V$. 
A~\emph{measurable fundamental domain} of the torus $V/\Lambda$ is a measurable set $P \subset V$ that satisfies 
\[
P +\Lambda = V \mbox{ and }  \operatorname{Leb}_V(P\cap(P+\bn))=0 \mbox{  for all } \bn\in\LL\setminus\{0\}.
\]
Let $P$ be a measurable fundamental domain of the torus $V/\Lambda$.
We consider the translation 
\[
T_{\balpha} :\,  V/\Lambda \to V/\Lambda , \ x \mapsto  x +\balpha \mod \Lambda,
\]
which we assume to be minimal.
A~collection $\{P_1,\ldots, P_h\}$ is said to be a \emph{natural partition} (it is a partition up to zero measure sets) of~$P$ with respect to~$T_{\balpha}$ if  
\begin{enumerate}
\itemsep.5ex
\item
$\bigcup_{i=1}^h P_i = P$;
\item
$\operatorname{Leb}_V(P_i \cap P_j) = 0$  for all $i \ne j$, $1\le i,j \le h$;
\item
each $P_i$, $1\le i \le h$, is the closure of its interior and $\operatorname{Leb}_V(\partial P_i)=0$;
\item
there exist $\bt_1, \dots, \bt_h \in V$ with  $\bt_i \equiv \balpha \bmod \Lambda$ such that $\bt_i+ P_i \subset P$, $1 \le i \le h$. 
\end{enumerate}
A~natural partition is called \emph{bounded} if the set~$P$ is bounded. 
\end{definition}

In the following, we shall consider fundamental domains and natural partitions  in the vector space $V = \{(x_1,\dots,x_d) \in \mathbb{R}^d : \sum_{i=1}^d x_i = 0\}$, with  $\Lambda = V\cap\Z^d$, $V/\Lambda$ being isomorphic to~$\T^{d-1}$, and we shall  also project these fundamental domains and natural partitions onto~$\mathbb{R}^{d-1}$.
Note that we keep the same notation as before for~$T_{\balpha}$ in order to simplify the notation.
Note also that we focus here on partitions by finite unions of convex polytopes  with nonempty interiors (e.g.\ as in Figure~\ref{f:finiteunionsbis}), hence Condition~(3) is satisfied. 
Recall  also that interiors of subsets of $V$  are considered with respect to the ambiant space~$V$.

Such a natural partition $\{P_1,\dots, P_h\}$  allows us  to define a.e.\ on~$P$ a map $\tau_{\bt}: P \to P$ as an \emph{exchange of domains} (which depends on the partition) by 
\begin{equation} \label{eq:tau}
\tau_{\bt}(\bx) = \bx + \bt_i \quad \mbox{whenever}\ \bx \in \mathring{P_i}.
\end{equation}
The map~$\tau_{\bt}$ is defined on $P \setminus \bigcup_{i=1}^h \partial P_i $, hence, it is defined almost everywhere. 
The dynamical system $(P, \tau_{\bt}, \operatorname{Leb}|_P)$  is measurably conjugate to $(V/\Lambda, T_{\balpha})$ (endowed with the Haar measure). 
One has for a.e.\ $\bx \in P$, $\tau_{\bt}(\bx) \equiv T_{\balpha}(\bx)  \bmod \Lambda$.  
The collection $\{P_1{+}\bt_1, \dots, P_h{+}\bt_h\}$ also forms a measurable natural partition of~$P$, hence the terminology exchange of domains. 

Now that we have seen how  exchanges of pieces act as toral translations, we discuss their symbolic codings. 
We continue with the notation of Definition~\ref{def:coding}.
 
\begin{definition}[Natural coding] \label{def:nc}
A~sequence $(u_n)_{n\in\mathbb{N}} \in \{1,\dots,h\}^{\mathbb{N}}$ is said to be a \emph{natural coding} of the minimal toral translation $(V/\Lambda, T_{\balpha})$ w.r.t.\ the natural partition $\{P_1,\dots,P_h\}$ if there exists $\bx \in P$ such that $(u_n)_{n\in\mathbb{N}}$ codes the orbit of~$\bx$ under the action of~$\tau_{\bt}$, i.e., 
\[
\tau_{\bt} ^n(\bx) = \bx + \sum_{k=0}^{n-1} \bt_{u_k} \in \mathring{P}_{i_n}
\]
for all $n \in \mathbb{N}$; note that $T_{\balpha}^n(\bx)  \equiv \tau_{\bt}^n(\bx) \bmod  \Lambda$. 
If $u$ is a natural coding of $(V/\Lambda, T_{\balpha})$ w.r.t.\ a natural partition $\{P_1,\dots,P_h\}$ whose elements $P_1, \dots, P_h$ are bounded, we call $u$ a \emph{bounded natural coding}.
\end{definition}

\begin{remark}
Bédaride and Bertazzon \cite{BB:2013} have shown that a natural partition associated with a minimal translation~$T_{\balpha}$ of~$\mathbb{T}^{d-1}$ has at least $d$ pieces, hence the alphabet of a natural coding has at least $d$ letters.
\end{remark}

\begin{remark} \label{r:differentcoding}
When $P$ is bounded, the natural codings of two different points $\bx, \by \in P$ cannot be equal. Indeed, if they were equal, we would have by definition  $\tau_{\bt}^n(\by) =\tau_{\bt}^n(\bx){+}\by{-}\bx$ for all $n\in\N$, but using the minimality of the toral translation~$T_{\balpha}$, we can see that, for any $\bu \neq \mathbf{0}$, there exists some $n \in \N$ such that $\tau_{\bt}^n(\bx){+}\bu \notin P$, so that $\tau_{\bt}^n(\by) = \tau_{\bt}^n(\bx){+}\by{-}\bx \notin P$ for some~$n$, a~contradiction. 
\end{remark}

The shift $(X_u,S)$ generated by a natural coding~$u$ of the minimal translation $(V/\Lambda, T_{\balpha})$ is minimal, uniquely ergodic, and has purely discrete spectrum according to \cite[Lemma~5.12]{BST:20} or
\cite[Theorems~A and~B]{Chevallier}. Hence the  two main steps in the proof of Theorem~\ref{theo:main} are, firstly, to exhibit the partition providing a natural coding
(see Section~\ref{sec:cons}) and, secondly, to estimate the factor complexity (see Sections~\ref{subsec:idea} to~\ref{subsec:fi}).

\subsection{Hypercubic fundamental domains} \label{subsec:hfd}
We now illustrate the formalism developed in the previous section with a very simple choice of a fundamental domain for~$\mathbb{T}^{d-1}$, following e.g.\ \cite{Ar.Ma.Sh.Ta.94}, see also \cite{Bar95}.
This will provide a simple geometric model for~$T_{\balpha}$ defined in  \eqref{eq:talpha} as an exchange of pieces, that will play a crucial role in Section~\ref{sec:cons}. 

Define
\begin{equation}\label{def:F}
F_i = \{(x_1,\dots,x_d) \in [0,1]^d:\, x_i = 1\} \qquad (1 \le i \le d)
\end{equation}
to be an \emph{upper face} of the $d$-dimensional unit cube,  and similarly 
\[ 
\tilde{F}_i = \{(x_1,\dots,x_d) \in [0,1]^d:\, x_i = 0\}
\]
as a \emph{lower face}.  
As an illustration, the union of the three lower faces when $d=3$ is depicted as
\includegraphics[height=4.5mm]{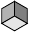}, see also Figure~\ref{f:0}, and Figure~\ref{fig:exchange} for $d=2$. 
Let $\bone$ be the vector all of whose entries equal one and define its orthogonal complement 
\[
\bone^\perp := \Big\{(x_1,\dots,x_d) \in \mathbb{R}^d \,:\, \sum_{i=1}^d x_i = 0\Big\}.
\]
Denote by 
\[
\pi_{\balpha}:\, \mathbb{R}^d \to \bone^\perp
\mbox{ the  projection along  } \mathbb{R} \balpha \mbox{ onto } \bone^\perp , 
\]
i.e., for $1\leq i \leq d$,   
\begin{equation}  \label{eq:pi}
\pi_{\balpha}(\be_i) = \be_i-\balpha.
\end{equation}
Let 
\begin{equation} \label{def:E}
E_{\balpha}= \bigcup_{i=1}^d E_{\balpha,i} \quad \mbox{with} \quad E_{\balpha,i} = \pi_{\balpha}(F_i).
\end{equation}
First observe that 
\[
E_{\balpha}= \pi_{\balpha}([0,1]^d).
\]
Indeed, we have $E_{\balpha}\subseteq \pi_{\balpha}([0,1]^d)$, and for $\bx = \sum_{i=1}^d x_i \pi_{\balpha}(\be_i)$, $x_i \in [0,1]$, we use that $\sum _{i=1} ^d \alpha_i \pi_{\balpha}(\be_i)=\pi_{\balpha}(\balpha)=\mathbf{0}$
to obtain that, for $j$ satisfying $\frac{1-x_j}{\alpha_j} = \min\{\frac{1-x_i}{\alpha_i} : 1 \le i \le d\}$,
\[
\bx = \sum_{i=1} ^d x_i \pi_{\balpha}(\be_i) + \frac{1{-}x_j}{\alpha_j} \sum_{i=1} ^d \alpha_i  \pi_{\balpha}(\be_i) = \sum_{i=1} ^d \Big(x_i + \frac{1{-}x_j}{\alpha_j} \alpha_i\Big) \pi_{\balpha}(\be_i) \in \pi_{\balpha}(F_j) \subset E_{\balpha}.
\]
We also obtain that if $\bx \in F_k \setminus F_j$, then with $t=\frac{1-x_j}{\alpha_j}$, we have $\bx{+}t\balpha \in (\be_j{+}\be_j^\perp) \setminus F_j$, which implies that $\pi_{\balpha}(\bx) = \pi_{\balpha}(\bx{+}t\balpha) \notin \pi_{\balpha}(F_j)$. Therefore,
\[
E_{\balpha,j} \cap E_{\balpha,k} = \bigg\{\sum_{i=1} ^d x_i\, \pi_{\balpha}(\be_i) \,:\, x_j = x_k = 1,\, x_i \in [0,1] \ \mbox{for}\ i \notin \{j,k\}\bigg\}
\]
is a ($d{-}2$)-dimensional subset of~$\bone^\perp$, hence $\bigcup_{i=1}^d E_{\balpha,i}$ forms a \emph{topological  partition} of~$E_{\balpha}$, i.e., the atoms~$E_{\balpha,i}$ have disjoint interiors.

We then consider the polyhedral exchange map
\begin{equation}\label{def:exchange-pieces}
\tilde{T}_{\balpha}:\, E_{\balpha}\to E_{\balpha}, \quad \bx \mapsto \bx + \balpha - \be_i \quad \mbox{if}\ \bx \in E_{\balpha,i}.
\end{equation}
Here and in the following, we neglect the intersections $E_{\balpha,i} \cap E_{\balpha,j}$, $i \ne j$.
Then we have
\[
\tilde{T}_{\balpha}(E_{\balpha,i}) = E_{\balpha,i} - \pi_{\balpha}(\be_i) = \pi_{\balpha}(\tilde{F}_i).
\]
Since the $\tilde{F}_i$'s are lower faces of $[0,1]^d$, we have $E_{\balpha}= \bigcup_{i=1}^d \tilde{T}_{\balpha}(E_{\balpha,i})$, hence the map $\tilde{T}_{\balpha}$ is an \emph{exchange of pieces} (in~$\bone^\perp$); see Figure~\ref{fig:exchange} for $d=2$ as an illustration.
 
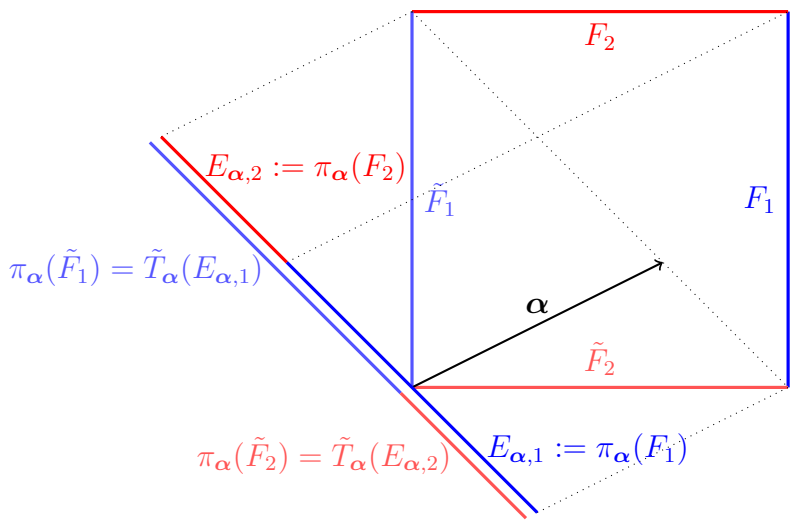
\begin{figure}
 \centerline{\begin{tikzpicture}[scale=5]
\draw[blue,very thick](1,0)--node[left]{$F_1$}(1,1);
\draw[red,very thick](0,1)--node[below]{$F_2$}(1,1);
\draw[blue!67,very thick](0,0)--node[right]{$\tilde{F}_1$}(0,1);
\draw[red!67,very thick](0,0)--node[above]{$\tilde{F}_2$}(1,0);
\draw[thick,->](0,0)--node[above]{$\balpha$}(.667,.333);
\draw[dotted] (0,1)--(1,0) (-.667,.667)--(0,1) (-.333,.333)--(1,1) (.333,-.333)--(1,0);
\draw[blue,very thick](.333,-.333)--node[blue,right,pos=.25]{$E_{\balpha,1}:=\pi_{\balpha}(F_1)$}(-.333,.333);
\draw[red,very thick](-.667,.667)--node[red,right,pos=.25]{$E_{\balpha,2}:=\pi_{\balpha}(F_2)$}(-.333,.333);
\begin{scope}[shift={(-.03,-.015)}]
\draw[blue!67,very thick](0,0)--node[left]{$\pi_{\balpha}(\tilde{F}_1) = \tilde{T}_{\balpha}(E_{\balpha,1})$}(-.667,.667);
\draw[red!67,very thick](0,0)--node[left]{$\pi_{\balpha}(\tilde{F}_2) = \tilde{T}_{\balpha}(E_{\balpha,2})$}(.333,-.333);
\end{scope}
\end{tikzpicture}}
\caption{Exchange of pieces, $d=2$.}\label{fig:exchange}
\end{figure}

Since $\{\bx {+} \pi_{\balpha}([0,1]^d) : \bx \in \mathbb{Z}^d \cap
\bone^\perp\}$ forms a tiling of~$\bone^\perp$, the set~$E_{\balpha}$ is a
measurable fundamental domain of $\bone^\perp / (\mathbb{Z}^d \cap
\bone^\perp)$. 
Indeed, let us first  show that the translates of $E_{\balpha}$ cover~$\bone^\perp$.  Let $\mathcal{Q}^+$ be the union of all the unit lattice hypercubes  that are included in the closed half space $H^+ = \{\bx \in \mathbb{R}^d : l(\bx)=x_1{+}\cdots{+}x_d \geq 0\}$. Since $\pi_{\balpha}(\partial \mathcal{Q}^+)  =\bone^\perp$, it is enough to show that the boundary~$\partial \mathcal{Q}^+$ is included in the union of all the lower faces of  the unit lattice  hypercubes contained in~$H^+$ with one vertex in~$\bone^\perp$.
On the one hand, $\partial \mathcal{Q}^+$~is contained in~$H^+$ and is a union of $(d{-}1)$-dimensional faces of unit lattice hypercubes. On the other hand, $\partial \mathcal{Q}^+= \partial \mathcal{Q}^-$ where $\partial \mathcal{Q}^-$ is the union of all the unit lattice  hypercubes not included in~$H^+$, and therefore $\partial \mathcal{Q}^+$ is a union of faces of  unit lattice  hypercubes with at least one vertex in $\{\bx \in \mathbb{R}^d : l(\bx) \leq -1\}$, which in turn implies that $\partial \mathcal{Q}^+ \subset \{\bx \in \mathbb{R}^d : l(\bx)\leq d{-}1\}$. It follows that each face of~$\partial \mathcal{Q}^+$ has a vertex in~$\bone^{\perp}$. This shows that the translates of~$E_{\balpha}$ cover~$\bone^\perp$.
Next, each face $\bx{+}\tilde{F_i}$ with $\bx \in \bone^\perp \cap \Z^d$ and $1 \le i \le d$ is actually in $\partial \mathcal{Q}^+$, because $\bx{+}\tilde{F_i}$ is also a face of the cube $\bx{-}\be_i{+}[0,1]^d \subset \mathcal{Q}^-$. Lastly, it is easy to see that a half-line $\bx{+}[0,\infty)\balpha$ with $\bx \in \mathcal{Q}^+$ is entirely included in~$\mathcal{Q}^+$ and that a half-line $\bx{-}[0,\infty)\balpha$ with $\bx \in \mathcal{Q}^-$ is entirely  included in~$\mathcal{Q}^-$. 
Since the intersection of a line $\bx {+} \mathbb{R} \balpha$ with the boundaries of the lattice cubes is a discrete set, a line $\bx {+} \mathbb{R} \balpha$ can meet $\partial \mathcal{Q}^+$ in at most one point. 
It follows that $\pi_{\balpha}$ is one to one on $\partial \mathcal{Q}^+$.

To obtain the toral translation~$T_{\balpha}$ (see~\eqref{eq:talpha}) from~$\tilde{T}_{\balpha}$ (defined in~\eqref{def:exchange-pieces}), we now omit the last coordinate, i.e.,  we consider the conjugation by
\[
\iota:\, \bone^\perp \to \mathbb{R}^{d-1}, \quad (x_1,\dots,x_d) \mapsto (x_1,\dots,x_{d-1}).
\]
Then $\iota(E_{\balpha})$ is a measurable fundamental domain of $\mathbb{R}^{d-1} / \mathbb{Z}^{d-1}$, and $\iota \circ \tilde{T}_{\balpha} = T_{\balpha} \circ \iota$.
Moreover,  the collection $\{\iota(E_{\balpha,1}), \dots, \iota(E_{\balpha,d} )\}$ is  a  bounded natural partition with respect to~$T_{\balpha}$, according to Definition~\ref{def:coding}. Indeed, the map $\iota \circ \tilde{T}_{\balpha}$ coincides with  the map~$\tau_{\bt}$ from \eqref{eq:tau} with vectors $\bt_1 = (\alpha_1{-}1, \alpha_2, \dots, \alpha_{d-1})$, \dots,  $\bt_{d-1} = (\alpha_1, \dots, \alpha_{d-2}, \alpha_{d-1}{-}1)$,  and $\bt_d = ( \alpha_1, \alpha_2,  \dots, \alpha_{d-1})$. They all satisfy $\bt_i \equiv (\alpha_1, \dots, \alpha_{d-1}) \mod \mathbb{Z}^{d-1}$.

In fact, we have shown the following proposition.

\begin{proposition} \label{p:0}
Let $\balpha = (\alpha_1,\dots,\alpha_d) \in (0,1)^d$ be a  totally irrational frequency vector.
Then $\iota(E_{\balpha})$ is a measurable fundamental domain of~$\mathbb{T}^{d-1}$ admitting the bounded natural partition $\iota(E_{\balpha}) = \bigcup_{i=1}^d \iota(E_{\balpha,i})$. 
Moreover, the map $\tilde{T}_{\balpha}: E_{\balpha} \to E_{\balpha}$ is conjugate to the minimal translation~$T_{\balpha}$ on~$\mathbb{T}^{d-1}$.  
\end{proposition}

The above proof shows that $\iota(E_{\balpha})$ is a measurable fundamental domain even if $\alpha_1,\dots,\alpha_d$ are linearly dependent over~$\mathbb{Q}$.
See Figure~\ref{f:0} for an illustration of~$\iota(E_{\balpha})$ and the image by~$\tilde{T}_{\balpha}$ for $d=3$. 

\begin{figure}[ht] 
\begin{tikzpicture}[scale=3.75]
\newcommand{\alphaone}{.48}
\newcommand{\alphatwo}{.32}
\newcommand{\alphathree}{.2}
\newcommand{\toright}{2.1}
\newcommand{\TEone}{\filldraw[fill=blue!10](0,0)--(-\alphaone,-\alphatwo)--(-2*\alphaone,1-2*\alphatwo)--(-\alphaone,1-\alphatwo)--cycle;}
\newcommand{\TEtwo}{\filldraw[fill=red!10](0,0)--(-\alphaone,-\alphatwo)--(1-2*\alphaone,-2*\alphatwo)--(1-\alphaone,-\alphatwo)--cycle;}
\newcommand{\TEthree}{\filldraw[fill=green!10](0,0)--(1-\alphaone,-\alphatwo)--(1-2*\alphaone,1-2*\alphatwo)--(-\alphaone,1-\alphatwo)--cycle;} 
\newcommand{\coord}{\draw[->](-.95,0)--(1.1,0); \draw[->](0,-.7)--(0,1.1); \draw(1,-.02)node[below]{$1$}--(1,.02) (-.02,1)node[left]{$1$}--(.02,1);}
\begin{scope}[shift={(-\toright+1-\alphaone,-\alphatwo)}]
\TEone \node at (-\alphaone,.5-\alphatwo){$\iota(E_{\balpha,1})$};
\end{scope}
\TEone \node at (-\alphaone-.07,.5-\alphatwo){$\iota(E_{\balpha,1}{+}\balpha{-}\be_1)$};
\begin{scope}[shift={(-\toright-\alphaone,1-\alphatwo)}]
\TEtwo \node at (.5-\alphaone,-\alphatwo){$\iota(E_{\balpha,2})$};
\end{scope}
\TEtwo \node at (.5-\alphaone,-\alphatwo){$\iota(E_{\balpha,2}{+}\balpha{-}\be_2)$};
\begin{scope}[shift={(-\toright-\alphaone,-\alphatwo)}]
\TEthree \node at (.5-\alphaone,.5-\alphatwo){$\iota(E_{\balpha,3})$};
\end{scope}
\TEthree \node at (.65-\alphaone,.5-\alphatwo){$\iota(E_{\balpha,3}{+}\balpha{-}\be_3)$};
\coord
\begin{scope}[shift={(-\toright,0)}]
\coord
\end{scope}
\end{tikzpicture}
\caption{The ($\iota$-representation of the) parallelepipeds $E_{\balpha,i}$ and their $\tilde{T}_{\balpha}$-images for $d=3$, $\balpha = (0.48,0.32,0.2)$.} \label{f:0}
\end{figure}
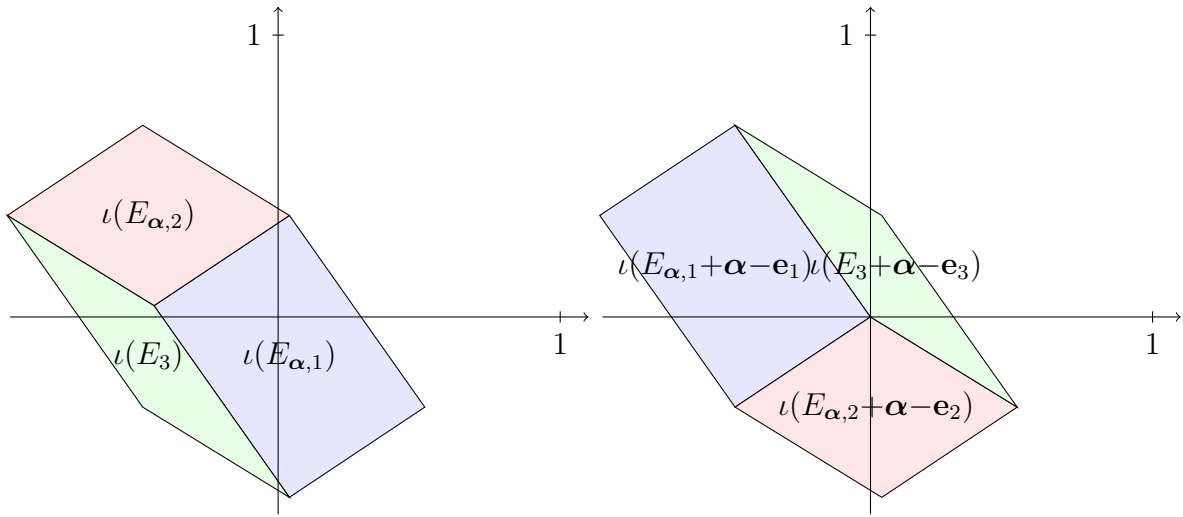

\subsection{Construction of natural partitions} \label{sec:constr-natur-part}
We have considered in the previous section a very simple choice of a fundamental domain. 
In the present section, we consider a more general situation, and we prove two technical propositions that allow the construction of a measurable fundamental domain
(Proposition~\ref{prop:funddomain}),  and of  a natural partition with respect to a toral translation (Proposition~\ref{prop:partition}). This rather general construction will be used with Tijdeman sequences in Section~\ref{subsec:tijdeman}.  Among others, a~technical aspect comes from the fact that we are handling maps that are not continuous, but only piecewise continuous; see Remark~\ref{rem:orbit} below.

Similarly to Definition~\ref{def:coding}, we consider a finite dimensional real vector space~$V$, a full rank lattice~$\LL$ in~$V$ and $\balpha \in V$.   
Let $R$ and $R_1,\dots,R_h$ be  nonempty subsets of~$V$ such that  $R = R_1\cup R_2\cup\dots\cup R_h$. 
For each $i\in\{1,\dots,h\}$, fix some $\bn_i \in \LL$. Consider the maps 
\[
T_i:\, V \fff V, \quad \mbox{and} \quad \widehat{T}:\, R\fff V 
\] defined by  $T_i(\bx)=\bx+\balpha+\bn_i$ and $\widehat{T}(x)=T_i(x)$ when $ \bx\in R_i$. If   $\bx$ belongs to more than one $R_i$, a single index $i=i(\bx)$ is  set once and for all, chosen from indices such that $\bx\in R_i$. A possible choice for  $i(\bx)$ is  the smallest $i$ for which  $\bx\in R_i$.
Finally, let $\bx_0\in R$. 
We  make the following  assumptions.

\begin{enumerate}
	\itemsep.5ex
	\item \label{i:a1}
	The  translation $T_{\balpha}:  V/\Lambda \rightarrow V/\Lambda$, $x \mapsto  x {+} \balpha \mod \Lambda$ is minimal;
	\item \label{i:a2}
	the sets~$R_1\dots,R_h$ are closed and $T_i(R_i) \subset R$ for all $i\in\{1,\dots,h\}$;
	\item \label{i:a3} 
	$\operatorname{Leb}_V(R_i \cap R_j) = 0$ for all $i \neq j \in \{1,\dots,h\}$;
	\item \label{i:a4} 
	there exist a compact set $K \subset R$ and a nonempty open set $U \subset K$ such that $K$ contains the orbit $\{\widehat{T}^n(\bx_0) : n\in\N\}$, and such that any point $\bx \in U$ has only one representative $\bmod\ \Lambda$ in~$K$, i.e., $\bx+\bn\notin K$ for all $\bn \in \LL \setminus \{\mathbf{0}\}$;
	\item \label{i:a5} 
	for each $i \in \{1,\dots,h\}$, $R_i$ is a finite union of convex polytopes with nonempty interiors.	
\end{enumerate} 

The following proposition allows us to cope with the fact that~$\widehat{T}$ is only piecewise continuous. A~crucial  argument in the proof is that, thanks to Assumptions~(\ref{i:a1}) and~(\ref{i:a4}), any $\widehat{T}$-backward orbit that lies in~$K$ must enter the open set~$U$.  

\begin{remark} \label{rem:orbit}
	If $(X,T)$ is a compact continuous dynamical system and if $x\in X$, then any point $y\in\overline{\{T^n(x):n\in\mathbb N\}}\setminus\{x\}$ is  the image by~$T$ of some element in $\overline{\{T^n(x):n\in\mathbb N\}}$; in other words, the image by $T$ of the closure of the orbit  of $x$ contains the closure of the  orbit  except perhaps~$x$. This property no longer holds in the following case with $T$ not being  continuous. Let $T: [0,1] \fff [0,1]$ be defined by $T(x) = x{+}\alpha$ for $x \in [0,1{-}\alpha]$ and $T(x) = x{+}\alpha{-}1$ for $x \in (1{-}\alpha,1]$, where $\alpha \in [0,1] \setminus \Q$. One has $0 \notin T(X)$, whereas the orbit closure of any point is~$X$. However, the following proof shows that the image of the  closure of the orbit always contains the closure of the orbit up to a negligible set.
\end{remark}

Let \[D = \overline{\{\widehat{T}^n(\bx_0) : n \in \N\}}.\]

\begin{proposition} \label{prop:funddomain} 
	If Assumptions (\ref{i:a1})--(\ref{i:a4}) hold, then $D $ is a measurable fundamental domain of the torus~$V/\Lambda$.
\end{proposition}

\begin{proof}	
Since $D$ is contained in the compact set $K$, the projection of~$D$ onto the torus~$V/\LL$ is a compact set that contains the sequence of projections of the points $\widehat{T}^n(\bx_0)$, $n\in\N$, hence the projection of~$D$ is the whole torus by the minimality assumption from~(\ref{i:a1}). It follows that $\bigcup_{\bn\in\LL} (D{+}\bn) = V$. 
	
We now want to prove that $\by{-}\by' \in \LL$ implies $\by = \by'$ for all $\by,\by' \in D \setminus \mathcal{N}$, where $\mathcal{N}$ is Lebesgue-null. 
If we find a  Lebesgue-null set~$\mathcal{N}$ such that for all $\by \in D \setminus \mathcal{N}$ and all $n \in\N$, there exists $\bx \in D$ such that $\widehat{T}^n(\bx) = \by$, then we are done. Indeed, let $\by,\by' \in D \setminus \mathcal{N}$ be such that $\by{-}\by' \in \LL$.  Since the translation~$T_{\balpha}$ is minimal and since $U$ is a nonempty open set, there exists $m \in\N$ such that $\by{-}m\balpha \in U{+}\LL$. By our assumption on~$\mathcal{N}$, there exist $\bz,\bz' \in D$  such that $\widehat{T}^m(\bz) = \by$ and $\widehat{T}^m(\bz') = \by'$. 
By definition of~$\widehat{T}$, we have $\bz = \by{-}m\balpha \bmod \LL$ and $\bz' = \by'{-}m\balpha \bmod \LL$, so that $\bz{-}\bz'\in \LL$. Now, $\bz \in U {+} \LL$ and $\bz,\bz' \in D \subset K$, hence by Assumption~(\ref{i:a4}), $\bz = \bz'$ which in turn implies $\by = \by'$.
	
It remains to define~$\mathcal{N}$. Let
\[
\mathcal{N} = \bigcup_{k\geq 0} \widehat{T}^k{\mathcal{N}_0}, \quad \mbox{with} \quad \mathcal{N}_0 = \big\{\widehat{T}^n(\bx_0) : n\in\N\big\} \cup \bigcup_{i,j\in\{1,\dots,h\},\,i\neq j} T_i\big(R_i\cap R_j\big).
\] 
Thanks to Assumption~(\ref{i:a3}), $\mathcal{N}$~is a null set.
Let us show now that, if $\by \in D \setminus \mathcal{N}_0$, then $\widehat{T}(\bx) = \by$ for some $\bx \in D$.
Since $\by \notin \bigcup_{i\neq j} T_i(R_i\cap R_j)$, we have $\eps = \dd(\by, \bigcup_{i\neq j} T_i(R_i\cap R_j)) > 0$.
Since $\by \neq \widehat{T}^n(\bx_0)$ for all $n \in \N$, there exists an increasing sequence of integers $(n_k)_k$ such that $\by = \lim_{k\fff\infty} \widehat{T}^{n_k}(\bx_0)$ and $\widehat{T}^{n_k}(\bx_0) \in B(\by,\eps/2)$ for all integers~$k$.
By passing to a subsequence, we can suppose that $\widehat{T}^{n_k}(\bx_0) = T_{i_0}(\widehat{T}^{n_k-1}(\bx_0))$ for all~$k$ and some fixed $i_0 \in \{1,\dots,h\}$.
For all  $k \in\N$ and all $j \in \{1,\dots,h\}$ with $j \neq i_0$, since $\widehat{T}^{n_k}(\bx_0) \in B(\by,\eps/2)$ and $\dd(\by, \bigcup_{i\neq j} T_i(R_i\cap R_j)) = \eps$, we have $\dd(T_{i_0}(\widehat{T}^{n_k-1}(\bx_0)),T_{i_0}(R_{i_0}\cap R_j)) \geq \eps/2$, hence $\dd(\widehat{T}^{n_k-1}(\bx_0), R_{i_0}\cap R_j) \geq \eps/2$. 
Again by passing to a subsequence, we can suppose that $\lim_{k\fff\infty} \widehat{T}^{n_k-1}(\bx_0) = \bx \in  R_{i_0}$. Since $\dd(\bx,R_{i_0}\cap R_j) \geq \eps/2$ for all $j\neq i_0$, it follows that $\bx \in R_{i_0}\setminus\bigcup_{j\neq i_0}R_j$. Therefore, $\widehat{T}(\bx) = T_{i_0}(\lim_{k\fff\infty} \widehat{T}^{n_k-1}(\bx_0)) = \lim_{k\fff\infty } T_{i_0}( \widehat{T}^{n_k-1}(\bx_0)) = \by$.
	
By induction, we see that for all $k \in \N$ and all $\by \in D \setminus \bigcup_{i=0}^k \widehat{T}^i \mathcal{N}_0$, there exists $\bx \in D$ such that $\widehat{T}^{k+1}(\bx) = \by$. Therefore, for all $\by \in D \setminus \mathcal{N}$, a~backward $\widehat{T}$-orbit of~$\by$ is in~$D$.
\end{proof}

\begin{proposition}\label{prop:partition} 
Suppose  that Assumptions (\ref{i:a1})--(\ref{i:a5}) hold. Let
\[
P = \overline{\mathring{D}} \quad \mbox{and} \quad P_i = \overline{\mathring{P}\cap \mathring{R}_i},\, i \in\{1,\dots,h\}.
\] 
Then
\begin{enumerate}[(a)]
\itemsep.5ex
\item \label{i:np1}
$P$ is a measurable fundamental  domain of the torus~$V/\Lambda$;
\item \label{i:np2}
$P = \bigcup_{i\in\{1,\dots,h\}}P_i$.
\end{enumerate}
Moreover, for each $i\in \{1,\dots,h\}$ such that $P_i\neq\emptyset$, one has 
\begin{enumerate}[(a)]
\setcounter{enumi}{2}
\item \label{i:np3}
$P_i\subset  R_i$ and $T_i(P_i) \subset P$;
\item \label{i:np4}
$P_i$ is a finite union of convex polytopes with nonempty interiors;
\item \label{i:np5}
$P_i$ is the closure of its interior, $\operatorname{Leb}(\partial P_i) = 0$ and $\operatorname{Leb}(P_i \cap P_j) = 0$ for all $j\neq i$.
\end{enumerate}
Thus, $P = \{P_1, \dots, P_h\}$ is a natural partition with respect to~$T_{\balpha}$.	
\end{proposition}

\begin{remark}
We will use  the following observations several times: If $Q$ is a convex polytope with nonempty interior, then $\mathring{Q}$ is convex, $Q = \overline{\mathring{Q}}$ and $\operatorname{Leb}(\partial Q) = 0$, and if $Q'$ is another convex polytope with nonempty interior, then $\mathring{Q} \cap Q' \neq \emptyset$ implies $\mathring{Q} \cap \mathring{Q'} \neq \emptyset$, which in turn implies $\overline{\mathring{Q} \cap \mathring{Q}'} = Q \cap Q'$.
\end{remark}

\begin{proof}[Proof of Proposition~\ref{prop:partition}] \textit{Preliminaries.} \\
By Assumption~(\ref{i:a5}), for all $i\in\{1,\dots,h\}$, $ R_i$~is a finite union of convex polytopes~$R_{ij}$, $j\in J_i$, with nonempty interiors.
Let $\KK = \{R_{ij} : i\in\{1,\dots,h\},\, j\in J_i\}$.  
By Assumption~(\ref{i:a3}), the interiors of the~$R_i$ are disjoint, hence for each $K\in\mathcal{K}$ there is a unique integer $i(K)\in\{1,\dots,h\}$ such that $K \subset R_{i(K)}$. For a convex polytope~$Q$ with nonempty interior,~let  
\[
\KK(Q) = \{K \in\KK \,:\, \mathring{Q} \cap \mathring{K} \neq\emptyset\}.
\]
For any  convex polytope $Q \subset R$ with nonempty interior,  by the above remark we have
\[
\mathring{Q} =\bigcup_{K\in \KK(Q)} \mathring{Q} \cap K \quad \mbox{and} \quad Q = \bigcup_{K\in\KK(Q)} Q \cap K.
\]
Observe that $\operatorname{Leb}(T_{i(K)}(D\cap K)\setminus D) = 0$ for each $K\in\KK$. Indeed, $\operatorname{Leb}(T_{i(K)}(D\cap K) \setminus D) > 0$ would imply that $\operatorname{Leb}( T_{i(K)}(D \cap \mathring{K}) \setminus D)>0$ and, since $\mathring{K} \cap T_{i(K)}^{-1}(V\setminus D) $ is open, there would exist $n\in \N$ such that $\bx = \widehat{T}^n(\bx_0) \in D \cap \mathring{K}\cap  T_{i(K)}^{-1}(V\setminus D)$ and, since $\mathring K\subset \mathring R_{i(K)}$, $\widehat{T}^{n+1}(\bx_0) = \widehat{T}(\bx) = T_{i(K)}(\bx) \notin D$, a~contradiction.
	
\smallskip
Let $Q_0 \subset U$, with $U$ as in Assumption~(\ref{i:a4}), be a closed hypercube with nonempty interior. By minimality, there exists $N \in\N$ such that $\bigcup_{n=0}^N \widehat{T}^n(\mathring{Q}_0) + \LL = V$.
	
We define by induction a  sequence $(\mathcal{R}_n)_n$ of sets. Let $\mathcal{R}_0 = \{Q_0\}$. Suppose $\mathcal{R}_n$ is defined and let
\[
\mathcal{R}_{n+1} = \bigcup_{Q\in\mathcal{R}_n} \big\{T_{i(K)}(Q \cap K) : K\in \KK(Q)\big\}.
\]
	
Let us see that for all~$n$ each $Q \in \mathcal{R}_n$ is a convex polytope with nonempty interior. By induction, if $Q \in \mathcal{R}_n$ is a convex polytope with nonempty interior, then for each $K\in\KK(Q)$,  $K\cap Q=\overline{\mathring{Q} \cap \mathring{K}}$ is a convex polytope  with nonempty interior. Therefore, $T_{i(K)}(Q \cap K)$ is a  convex polytope with nonempty interior.
	
Let
\begin{equation} \label{eq:setR}
\mathcal{R} = \bigcup_{n=0}^N \mathcal{R}_n \qquad \mbox{and} \qquad P' = \bigcup_{Q\in\mathcal{R}} Q .
\end{equation}
Let us show again by induction that, for all $Q\in\mathcal{R}_n$, we have $Q \subset D$.
Let $Q \in\mathcal{R}_n$, $K \in \KK(Q)$ and $Q' = T_{i(K)}(K\cap Q)$. 
If $Q'$ were not included in~$D$, the set $T_{i(K)}(Q \cap K) \setminus D$ would have positive measure because $D$ is closed, but by the induction hypothesis $Q \subset D$, thus  $\operatorname{Leb}(T_{i(K)}(Q \cap K) \setminus D) \leq \operatorname{Leb}(T_{i(K)}(D \cap K) \setminus D)$, contradicting $\operatorname{Leb}(T_{i(K)}(D \cap K)\setminus D) = 0$. It follows that $P' \subset D$.
	
Finally, let $F_0=\bigcup_{i\neq j}R_i\cap R_j$ and $F=\bigcup_{n\geq 0}\widehat{T}^{-n}(F_0{+}\LL)$.
	
\smallskip\noindent
$\bullet$\  To see~(\ref{i:np1}), let us first show  by induction  that $\widehat{T}^n(\mathring{Q}_0\setminus F) \subset \bigcup_{Q\in\mathcal{R}_n} Q{+}\LL$.  
Let $\bx\in \mathring{Q_0}\setminus F$, let $\by=\widehat{T}^n(\bx)$ and let $\by' \in Q$, $Q \in \mathcal{R}_n$ be such that $\by'{-}\by \in \LL$. 
Since $Q = \bigcup_{K\in\KK(Q)} Q \cap K$, we have $\by' \in Q \cap K$ for some $K\in \KK(Q)$. Moreover, $\by'\notin F_0$ otherwise we would have $\by\in F_0{+}\LL$, which contradicts $\bx\notin \widehat{T}^{-n}(F_0{+}\LL)$.
It follows that 
\[
\widehat{T}(\by')=T_{i(K)}(\by') \in T_{i(K)}(Q \cap K) \in\mathcal{R}_{n+1},
\]
which in turn implies $\widehat{T}^{n+1}(\bx)=\widehat{T}(\by) \in T_{i(K)}(Q \cap K) + \LL\subset \bigcup_{Q\in\mathcal{R}_{n+1}} Q + \LL$.	
	
Next, thanks to the choice of~$N$ and the inclusion $\bigcup_{n=0}^N\widehat{T}^n(\mathring{Q}_0\setminus F) \subset P'{+}\LL$, we obtain
\[
V=\LL+\bigcup_{n=0}^N\widehat{T}^n(\mathring{Q}_0) = \LL + \bigcup_{n=0}^N\widehat{T}^n(\mathring{Q}_0\setminus F) \cup \bigcup_{n=0}^N \widehat{T}^n(F) \subset (P'{+}\LL) \cup \bigcup_{n=0}^N \big(\widehat{T}^n(F){+}\LL\big).
\]
Now, by Assumption~(\ref{i:a3}), $F$ is Lebesgue negligible, hence $V\setminus (P'{+}\LL)$ is Lebesgue negligible. Since $P'$ is closed and bounded, it implies that $P'{+}\LL=V$, hence $\operatorname{Leb}(P')\geq \/\operatorname{Vol}(V/\LL)$.
	
Finally, for all $n $ and $Q \in \mathcal{R}_n$, we have $Q \subset D$ and $\overline{\mathring{Q}} = Q$, therefore $P' \subset \overline{\mathring{D}} = P$. 
On the other hand, if $\mathring{D} \setminus P' \ne \emptyset$, then it has nonzero Lebesgue measure, which is not possible for $\operatorname{Leb}(P') \geq \operatorname{Vol}(V/\LL)$ and by Proposition~\ref{prop:funddomain}, $\operatorname{Leb}(D) = \operatorname{Vol}(V/\LL)$. It follows that 
\[
P= \overline{\mathring{D}}= P'=\bigcup_{Q\in\mathcal{R}}Q,
\] 
and that $P$ is a measurable fundamental domain.
	
\smallskip\noindent
$\bullet$\ To see~(\ref{i:np2}), let us show that $P = \bigcup_{i\in\{1,\dots,h\}} P_i$.  Let $Q \in \mathcal{R}_n$ for some $n \leq N$. For all $K\in\mathcal{K}(Q)$, we have $Q \cap K = \overline{\mathring{Q} \cap \mathring{K}} \subset \overline{\mathring{P} \cap \mathring{R}_{i(K)}} = P_{i(K)}$, therefore $Q\subset \bigcup_{i\in I}P_i$.
	
\smallskip\noindent
$\bullet$\ We verify~(\ref{i:np3}).  Clearly $P_i \subset P \cap R_i$. Next, to show that $T_i(P_i) \subset P$, it is enough to prove that $T_i(\mathring{P} \cap \mathring{R_i}) \subset P$.  Now for all $K\in \KK$, $\operatorname{Leb}(T_{i(K)}(D \cap K) \setminus D) = 0$ and $R_i = \bigcup_{K\in\KK:i(K)=i} K$, therefore $\operatorname{Leb}(T_i({D} \cap {R_i}) \setminus D) = 0$. Since $\operatorname{Leb}(D\setminus P) = 0$, it follows that $\operatorname{Leb}(T_i(P \cap R_i) \setminus P) = 0$, which implies that $T_i(\mathring{P} \cap \mathring{R_i})\subset P$.
	
\smallskip\noindent
$\bullet$\ To see~(\ref{i:np4}) and~(\ref{i:np5}), we have only to show that each~$P_i$ is a finite union of convex polytopes with nonempty interiors, for, this implies that $P_i$ is the closure of its interior and that $\operatorname{Leb}(\partial P_i) = 0$.  It is enough to show that
\[
\bigcup_{(K,Q)\in\mathcal{K}\times \mathcal{R}:\,i(K)=i, K\in\mathcal{K}(Q)} \mathring{Q} \cap \mathring{K}\subset\,\mathring{P} \cap \mathring{R_i} \subset \bigcup_{(K,Q)\in\mathcal{K}\times \mathcal{R}:\,i(K)=i, K\in\mathcal{K}(Q)} Q\cap K.
\]
The first inclusion is clear because $P=P'$. For the second inclusion, let $\bx \in \mathring{P} \cap \mathring{R_i}$ and let $Q \in \mathcal{R}$  containing~$\bx$.  Consider the set~$\mathcal{K}_{\bx}$ of $K \in \mathcal{K}$ such that $\bx \in K$. 
There is an open ball $B(\bx,r)$ with $r>0$ such that $B(\bx,r) \subset \mathring{P} \cap \mathring{R_i}$ and such that $B(\bx,r)\cap K=\emptyset$ for all $K\in\mathcal K\setminus\mathcal K_{\bx}$. 
The ball $B(\bx,r)$ is included in the union of the~$K$, $K \in \mathcal{K}_{\bx}$, therefore there exists $K \in \mathcal{K}_{\bx}$ such that $\mathring{Q} \cap \mathring{K} \neq \emptyset$. 
Since $\mathring{Q} \cap \mathring{K} \neq \emptyset$, we have $K\in\KK(Q)$ and we are done because $x\in Q\cap K$.
	
Finally, $\operatorname{Leb}(P_i \cap P_j)=0$ for all $j\neq i$ follows from Assumption~(\ref{i:a3}) and $P_i\subset R_i$.
\end{proof}

\begin{remark} \label{r:choice}
In the proof of Proposition~\ref{prop:partition}, we have shown that $P = P' = \bigcup_{Q\in\mathcal{R}} Q$. The definition of the set $\mathcal{R}$ in \eqref{eq:setR} depends only on the hypercube $Q_0$ included in~$U$ and on the restriction of $\widehat{T}$ to the interior of the~$R_i$.
This shows that the fundamental domain $P=\overline{\mathring{D}}$ depends neither on the choice of $\widehat{T}$ on the intersection $R_i\cap R_j$ nor on the initial point~$\bx_0$, provided that the orbit of $\bx_0$ is contained in~$K$. Whereas, $D$~might depend on the initial point: if there is a point $\by_0\in K\setminus D$ whose orbit is in $K$, then the set $\overline{\{\widehat{T}^n(\by_0):n\in \N\}}$ is not equal to~$D$.
\end{remark}

\section{Two constructions of sequences with small discrepancy} \label{sec:cons}

\subsection{Strategy}\label{subsec:strategy}
We now have gathered all that is needed in terms of notation and concepts for describing in detail the strategy for the constructions of sequences developed in the present section.

We represent an infinite word~$u$ as the set $\{\bp(u_{[0,n)}) : n \in \mathbb{N}\}$ of vertices of a broken line~$\mathbf{L}_u$; see Figure~\ref{fig:brokenline}.
Our aim is to control the supremum of the discrepancy vectors 
\[
n \balpha - \bp(u_{[0,n)}),\, n \in \mathbb{N}.
\] 
One notices that $n \balpha {-} \bp(u_{[0,n)})= {-}\pi_{\balpha}(\bp(u_{[0,n)}))$ and thus  $\Delta_{\balpha}(u) = \sup_n \| \pi_{\balpha}(\bp(u_{[0,n)}) \|_\infty$.
Therefore, the discrepancy~$\Delta_{\balpha}(u)$ can be seen as the (Hausdorff) distance between the broken line~$\mathbf{L}_u$ and the line $\mathbb{R} \balpha$ with respect to the seminorm $|\bx| = \|\pi_{\balpha}(\bx)\|_\infty$.
One can also notice that if a sequence $(\bx_n)_{n\geq 0}$  of points in~$\bone^\perp$ satisfies $\bx_{n+1}=\bx_n{+}\balpha{-}\be_{u_n}$ for all $n\in\N$, then  
\begin{equation}\label{eq:discrepance}
n\balpha{-}\bp(u_{[0,n)}) = \bx_n{-}\bx_0\ \mbox{for all}\ n \in \mathbb{N} \quad \mbox{and}\quad  \Delta_{\balpha}(u) = \sup_{n\in\N}\|\bx_n{-}\bx_0\|_\infty.
\end{equation}
The strength of Tijdeman's construction relies on the fact that $\overline{\{{-}\pi_{\balpha}(\bp(u_{[0,n)}) : n \in \mathbb{N}\}}$ forms a fundamental domain of $\bone^\perp / (\mathbb{Z}^d \cap \bone^\perp)$, and that we can partition this fundamental domain into atoms $\overline{\{{-}\pi_{\balpha}(\bp(u_{[0,n)}\big) :  u_n= i, \, n \in \mathbb{N}\}}$, for each letter~$i$.
This then allows us  to relate the dynamics of the shift with the dynamics of~$T_{\balpha}$.
The sequence~$u$ is then a bounded natural coding according to Definition~\ref{def:nc} (whose associated natural partition is obtained by application of the map~$\iota$).

Let $\balpha = (\alpha_1,\dots,\alpha_d) \in (0,1)^d$ be a totally irrational frequency vector. 
The first construction in Section~\ref{subsec:billiard} produces classical hypercubic billiard sequences as presented in \cite{Ar.Ma.Sh.Ta.94}. 
The second one, given in Section~\ref{subsec:tijdeman}, corresponds to Tijdeman's construction in \cite{Tijdeman:80} and is obtained by introducing  more specification stated in terms of lower and upper bounds for the  supremum norm of the discrepancy vectors.  
Both constructions are obtained by coding the same toral translation~$T_{\balpha}$ with respect to  finite partitions by polytopes.
In particular, Section~\ref{subsec:billiard}, which is devoted to hypercubic billiard sequences, aims at explaining that they do not have the lowest possible discrepancy (see Proposition~\ref{prop:eqbillard}) and to prepare the main construction from Section~\ref{subsec:tijdeman}, which can be seen as an improvement of the hypercubic billiard codings.

\subsection{Hypercubic billiard  sequences} \label{subsec:billiard}
In this section, we consider cutting words associated with the hypercubic billiard. 
In Proposition~\ref{prop:eqbillard}, we recall that their discrepancy is generally not minimal for $d>2$; see also \cite{Andrieu}. 
The description of these cutting words will help with the understanding of Section~\ref{subsec:tijdeman}, where we recall Tijdeman's construction. 
Tijdeman's construction can be seen as an improvement of the present construction  for hypercubic billiard sequences.

We follow the approach of \cite{Ar.Ma.Sh.Ta.94} adapted to the present context.
Travelling along a half line $\mathbf{L} = \bx {+} \mathbb{R}_+\balpha$, $\bx \in [0,1)^d$, one meets the faces of the unit hypercubes that are located at the grid defined by the set~$\mathbb{Z}^d$ of integer points.
The \emph{cutting word} $u = u_0u_1\cdots \in \{1,\dots,d\}^{\mathbb{N}}$ codes the sequence of upper faces (of unit hypercubes) that are met by~$\mathbf{L}$, where if the line hits a face parallel to~$\be_j^\perp$, we code this intersection with the letter~$j$; see Figure~\ref{fig:cutting} for an illustration.
More precisely, starting at~$\bx$ and given $u_0 \cdots u_{n-1}$, the letter~$u_n$ is the coding of the upper face of $\bp(u_{[0,n)}) {+} [0,1]^d$ (as defined in~\eqref{def:F}) that is intersected by $\bx {+} \mathbb{R} \balpha$. 
In the latter description, we can replace~$\bx$ by any point on $\bx {+} \mathbb{R} \balpha$, in particular by $\bx_0 = \pi_{\balpha}(\bx) \in E_{\balpha}$. 
We call such a coding sequence a \emph{hypercubic billiard sequence with frequency~$\balpha$ and initial condition~$\bx$ (or~$\bx_0$)}.

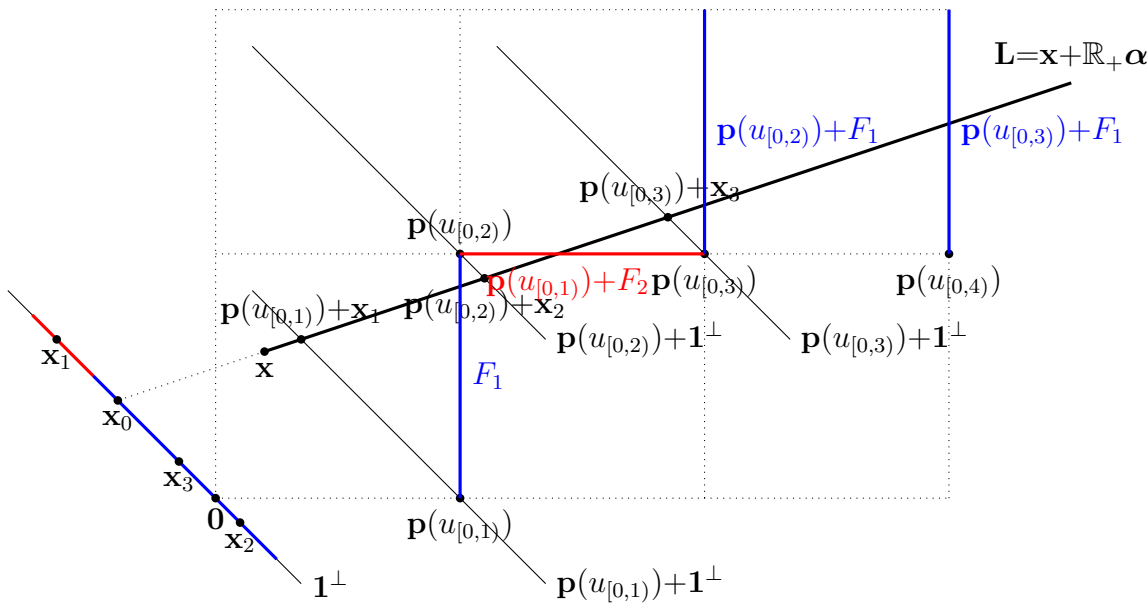
\begin{figure}
\centerline{\begin{tikzpicture}[scale=3.25]
\draw[dotted](0,0) grid (3,2);
\draw[dotted](.2,.6)--(-.4,.4);
\draw[very thick](.2,.6)--(3.5,1.7)node[above]{$\mathbf{L} {=} \bx {+} \mathbb{R}_{+} \balpha$};
\draw(-.85,.85)--(.35,-.35)node[right]{$\bone^\perp$};
\draw[red,very thick](-.5,.5)--(-.75,.75);
\draw[blue,very thick](-.5,.5)--(.25,-.25);
\fill(0,0) circle (.5pt)node[below]{$\mathbf{0}$};
\fill(.2,.6) circle (.5pt)node[below]{$\bx$};
\fill(-.4,.4) circle (.5pt)node[below]{$\bx_0$};
\fill(.35,.65) circle (.5pt)node[above]{$\bp(u_{[0,1)}){+}\bx_1$};
\fill(-.65,.65) circle (.5pt)node[below]{$\bx_1$};
\fill(1.1,.9) circle (.5pt)node[below=2pt]{$\bp(u_{[0,2)}){+}\bx_2$};
\fill(.1,-.1) circle (.5pt)node[below]{$\bx_2$};
\fill(1.85,1.15) circle (.5pt)node[above]{\hspace{-0.2cm}$\bp(u_{[0,3)}){+}\bx_3$};
\fill(-.15,.15) circle (.5pt)node[below]{$\bx_3$};
\fill(1,0) circle (.5pt)node[below]{$\bp(u_{[0,1)})$};
\fill(1,1) circle (.5pt)node[above]{$\bp(u_{[0,2)})$};
\fill(2,1) circle (.5pt);
\node[below] at (2.08,1){$\bp(u_{[0,3)})$};
\fill(3,1) circle (.5pt)node[below]{$\bp(u_{[0,4)})$};
\draw(.15,.85)--(1.35,-.35)node[right]{$\bp(u_{[0,1)}){+}\bone^\perp$};
\draw(.15,1.85)--(1.35,.65)node[right]{$\bp(u_{[0,2)}){+}\bone^\perp$};
\draw(1.15,1.85)--(2.35,.65)node[right]{$\bp(u_{[0,3)}){+}\bone^\perp$};
\draw[blue,very thick](1,0)--node[right]{$F_1$}(1,1);
\draw[red,very thick](1,1)--node[below]{$\bp(u_{[0,1)}){+}F_2$}(2,1);
\draw[blue,very thick](2,1)--node[right]{$\bp(u_{[0,2)}){+}F_1$}(2,2);
\draw[blue,very thick](3,1)--node[right]{$\bp(u_{[0,3)}){+}F_1$}(3,2);
\end{tikzpicture}}
\caption{Notation for hypercubic billiard sequences, with $\balpha \approx (0.75,0.25)$, $\bx \approx (0.2,0.6)$, thus $\bx_0 \approx (-0.4,0.4)$. {If $\bx_n\in   \{(-a,a): -\alpha_1< a<1-2\alpha_1 \}$, then $u_n=2$, and if  $\bx_n\in   \{(-a,a): 1-2\alpha_1< a< 1-\alpha_1\} $, then $u_n=1$}.  Here    $u_0u_1\cdots = 1211\dots$.} \label{fig:cutting}
\end{figure}

Note that the half line intersects the lattice~$\mathbb{Z}^d$ at most once since $\balpha$ is totally irrational; we neglect the starting points~$\bx$ for which $\mathbf{L} \setminus \{\bx\}$ may intersect the lattice~$\mathbb{Z}^d$.
We denote by~$\mathcal{B}_{\balpha}$ the set of hypercubic billiard sequences with frequency~$\balpha$ and initial condition~$\bx$ such that $\mathbf{L} \setminus \{\bx\}$ does not intersect the lattice~$\mathbb{Z}^d$.

For $\bx_0 \in E_{\balpha}$, we have $u_0 = i$ if $(\bx_0{+}\mathbb{R}_+\balpha) \cap  F_i \ne \emptyset$, i.e., if $\bx_0 \in E_{\balpha,i}$. 
For $n\geq 1$, set
\[
u_n = i \quad \mbox{if} \ (\bx_0 {+} \mathbb{R}_+\balpha) \cap (\bp(u_{[0,n)}) {+}  F_i) \ne \emptyset,
\]
i.e.,  $u_n = i$ if $\bx_0 - \pi_{\balpha}(\bp(u_{[0,n)})) \in E_{\balpha,i}$. 
By (\ref{eq:pi}), we have 
\[
\bx_n := \bx_0 - \pi_{\balpha}(\bp(u_{[0,n)})) = \bx_0 + n\balpha - \bp({u}_{[0,n)}).
\]

Since $u_n = i$ if $\bx_n \in E_{\balpha,i}$, we have by \eqref{def:exchange-pieces}  that  $\bx_{n+1} = \tilde{T}_{\balpha}(\bx_n) $ for all $n \in \mathbb{N}$. 
Thus $u$ is the coding of~$\bx_0$ w.r.t.~$\tilde{T}_{\balpha}$, which is the exchange of pieces w.r.t.\ the partition $\{E_{\balpha,i} : 1 \le i \le d\}$: one has $u_n=i$ if and only if $\tilde{T}_{\balpha}^n(\bx_0) \in E_{\balpha,i}$. Since $\iota$ is bijective, $u$~is also the bounded natural coding of~$\iota(\bx_0)$ w.r.t.\ the toral translation~$T_{\balpha}$  
and the partition $\{\iota(E_{\balpha,i}) : 1 \le i \le d\}$ (in the sense of Definition~\ref{def:nc}).

Note that $\bx_n {-} \bx_0 = n \balpha {-} \bp(u_{[0,n)})$ is a discrepancy vector of~$u$, thus the discrepancy is $\Delta_{\balpha}(u) = \sup_n \|\bx_n{-}\bx_0\|_\infty$.  
Writing $\bx_n = (x_{n,1},\dots,x_{n,d})$, and letting $t_{n,i} = \frac{1-x_{n,i}}{\alpha_i}$,
 we have 
\[
\bx_0 + (n {+} t_{n,i}) \balpha \in \bp(u_{[0,n)}) + \be_i+ \be_i^\perp.
\]
In other words, $t_{n,i}$ is the time needed on the line $\bx_0 {+} \mathbb{R} \balpha$ to go from $\bx_0 {+} n \balpha$ to the hyperplane $\bp(u_{[0,n)}) {+} \be_i {+} \be_i^\bot$.
We have $u_n = i$ if we hit $\bp(u_{[0,n)}) {+} F_i$, which is equivalent to 
\begin{equation}\label{eq:tnibilliard} 
t_{n,i} = \min\{t_{n,j} : 1 \le j\le d\}.
\end{equation}
This is to be compared to \eqref{eq:tijde} below, where the hyperplane to be hit will be of the form $\bp(u_{[0,n)}) {+} C \be_j {+} \be_j^\bot$. 
To construct Tijdeman sequences in Section~\ref{subsec:tijdeman},  we will optimize the choice of~$u_n$ with respect to two criteria expressed in terms of the values taken by the~$t_{n,i}$'s.   

\begin{remark}\label{rem:cps}
This construction can  be interpreted in terms of model sets with the acceptance window being given by $\bx_0{-}E_{\balpha}$; one has
\[
\{\bp(u_{[0,n)}) \,:\, n \in \mathbb{N}\} = \{\bx \in \mathbb{Z}^d \,:\, \pi_{\balpha}(\bx) \in \bx_0{-}E_{\balpha},\, \langle \bx, \bone^{\bot} \rangle \geq 0\},
\]
except possibly for points $\bx \in \mathbb{Z}^d$ with $\pi_{\balpha}(\bx)$ lying on the boundary of $\bx_0{-}E_{\balpha}$.
Indeed, we have proved above that each point $\bx = \bp(u_{[0,n)})$, $n \in \mathbb{N}$, satisfies $\pi_{\balpha}(\bx) \in \bx_0{-}E_{\balpha}$ and clearly $\langle \bx, \bone \rangle \geq 0$.  
The  reverse inclusion comes from the fact that $E_{\balpha}$ forms a fundamental domain of $\bone^\perp / (\mathbb{Z}^d \cap \bone^\perp)$, as shown in Section~\ref{subsec:hfd}.
The (half)  broken line~$\mathbf{L}_u$ associated with~$u$ is thus exactly the set of points $\bx \in \mathbb{N}^d$ such that
 $\pi_{\balpha}(\bx) \in \bx_0{-}E_{\balpha}$.
\end{remark} 

The next proposition provides estimates on the  discrepancy of hypercubic billiard sequences; see also \cite{Vuillon:03} expressed in terms of balance and \cite{Andrieu} for the case of particular~$\balpha$, where it is proved that for $d \geq 5$ and for every $k \in \{3,\dots,d{-}1\}$, there exists a hypercubic $k$-balanced billiard word (with a totally irrational frequency vector~$\balpha$).

We recall that $\mathcal{B}_{\balpha}$ stands for the set of hypercubic billiard sequences with frequency~$\balpha$ and initial condition $\bx_0 \in E_{\balpha}$ such that $\mathbf{L} \setminus \{\bx_0\}$  does not intersect the lattice~$\mathbb{Z}^d$.

\begin{proposition} \label{prop:eqbillard}
Let $\balpha = (\alpha_1,\dots,\alpha_d) \in (0,1)^d$ be a  totally irrational  frequency vector. 
Then
\[
\{\Delta_{\balpha}(u) \,:\, u \in \mathcal{B}_{\balpha}\} = \big[\tfrac12(1+(d{-}2)\|\balpha\|_{\infty}), 1 + (d{-}2) \|\balpha\|_\infty\big]. 
\]
Moreover, 
\[
\inf_{\balpha} \inf_{u\in \mathcal{B}_{\balpha}} \Delta_{\balpha}(u) = 1 - \tfrac{1}{d},
\]
where the infimum is taken over totally irrational  frequency vectors $\balpha \in (0,1)^d$.  
\end{proposition}
  
\begin{remark}\label{rem:billiardevenly}
We remark that, when $d \geq 3$, there exist hypercubic billiard sequences~$u$ in~$\mathcal{B}_{\balpha}$ that are fairly distributed, i.e., $\Delta_{\balpha}(u) \leq D_d = 1 {-} \frac{1}{2d-2}$. 
Indeed, it suffices to take $\balpha$ such that $\|\balpha\|_{\infty} $ is close to~$\frac{1}{d}$ and $\bx_0$ such that $\Delta_{\balpha}(u) $ is close to 
$\frac{1}{2}(1{+}(d{-}2)\|\balpha\|_{\infty})$.
\end{remark} 

\begin{remark}\label{rem:sturm}
When $d=2$, a fairly distributed sequence with totally irrational frequency~$\balpha$ is Sturmian; see Proposition~\ref{prop:sturm}.
Moreover, Proposition~\ref{prop:eqbillard} indicates that not all Sturmian sequences have the same discrepancy, although they are all $1$-balanced.
Indeed,  when $d=2$, Proposition~\ref{prop:eqbillard} gives  that the range of values taken by $\Delta_{\balpha}(u)$ is the whole segment $[1/2,1]$.
\end{remark}

\begin{proof}[Proof of Proposition~\ref{prop:eqbillard}]
Let $u \in {\mathcal B}_{\balpha}$ be a hypercubic  billiard sequence with a totally irrational  frequency vector~$\balpha$ and initial condition $\bx_0 \in E_{\balpha}$. 
Since $\tilde{T}_{\balpha}$ is minimal, the sequence $(\tilde{T}_{\balpha}^n(\bx_0))_n$  is dense in~$E_{\balpha}$ and 
\begin{equation} \label{eq:eqbillard}
\Delta_{\balpha}(u) = \sup_{n\in\mathbb{N}} \|n\balpha {-} \bp({u}_{[0,n)})\|_\infty = \sup_{n\in\mathbb{N}} \|\bx_n {-} \bx_0\|_\infty = \sup_{n \in\mathbb{N}} \|\tilde{T}_{\balpha}^n(\bx_0) {-} \bx_0\| = \sup_{\by\in E_{\balpha}} \|\by {-} \bx_0\|_\infty.
\end{equation}
Then elements of $E_{\balpha} = \pi_{\balpha}([0,1]^d)$ are of the form $\sum_{i=1} ^d  y_i (\be_i{-}\balpha)$, $y_i \in [0,1]$.
Therefore, the $i$-th coordinate is in $[{-}(d{-}1)\alpha_i, 1{-}\alpha_i]$, where the endpoints of the interval are attained by $\sum_{j\ne i}\pi_{\balpha}(\be_j)$ and $\pi_{\balpha}(\be_i)$ respectively. 
Therefore, the diameter of~$E_{\balpha}$ is $1{+}(d{-}2)\|\balpha\|_\infty$.
This implies that 
\[
\Delta_{\balpha} (u) \le 1 + (d{-}2) \|\balpha\|_\infty \le d-1
\]
and, by the triangle inequality, 
\[
\Delta_{\balpha}(u) \ge \frac{1+(d{-}2)\|\balpha\|_\infty}{2} \ge \frac{d-1}{d}.
\]
For $\bx_0 = \pi_{\balpha}(\be_i)$ such that $\alpha_i = \|\balpha\|_\infty$, we have $\Delta_{\alpha}(u) = 1 {+} (d{-}2) \|\balpha\|_\infty$ because the $i$-th coordinate of $\bx_0{-}\by$, $\by \in E_{\balpha}$, takes all values in $[0,1{+}(d{-}2)\alpha_i]$ and the other coordinates are in $[\alpha_j{-}1,(d{-}1)\alpha_j]$, $j \ne i$.
For $\bx_0 = \frac{1}{2} \pi_{\balpha}(\bone) = \frac{1}{2}(\bone{-}d\balpha)$, the $i$-th coordinate of $\by{-}\bx_0$, $\by \in E_{\balpha}$, takes all values in the interval $\big[{-}\frac{1}{2}(1{+}(d{-}2)\alpha_i),\frac{1}{2}(1{+}(d{-}2)\alpha_i)\big]$ for all $1 \le i \le d$, thus $\Delta_{\alpha}(u) = \frac{1}{2}(1 {+} (d{-}2) \|\balpha\|_\infty)$.
Now, taking convex combinations of $\pi_{\balpha}(\be_i)$ and $\frac{1}{2} \pi_{\balpha}(\bone)$ as initial condition~$\bx_0$, the discrepancy takes all values in the interval $\big[\frac12(1{+}(d{-}2)\|\balpha\|_{\infty}), 1{+}(d{-}2) \|\balpha\|_\infty\big]$.
Choosing $\balpha$ close to $(\frac{1}{d}, \dots, \frac{1}{d})$, we obtain that $\Delta_{\balpha}(u)$ can be arbitrarily close to~$1{-}\frac{1}{d}$.
\end{proof}

\subsection{Tijdeman's construction}\label{subsec:tijdeman}
Our aim now is to produce sequences~$u$ on a $d$-letter alphabet such that 
$$
\bx_0 {-} \pi_{\balpha}(\bp(u_{[0,n)})) \in [-C',C]^d  \mbox{ for all } n,
$$  for given constants $C,C'$.
In the previous section (with $C=1$, $C' \ge (d{-}1) \|\balpha\|_\infty$, see Remark~\ref{rem:billard} below), the word~$u$ was defined by the sequence of faces met by a half line~$\mathbf{L}$, which gave us the fundamental domain of Figure~\ref{f:0}.
In order to reduce the discrepancy, we work with different fundamental domains; see Figure~\ref{f:2} for an illustration.
More precisely, we modify the construction in order to obtain a fundamental domain in $[-C',C]^d$ for arbitrary $C,C'$ with 
\begin{equation} \label{e:CC}
C,C' \ge 1-\frac{1}{d}, \quad C \le 1, \quad \mbox{and} \quad C+C' \ge 2 - \frac{1+\min_i\alpha_i}{d-1},
\end{equation}
and we call these sequences Tidjeman sequences; see Definition~\ref{def:tijdeman} and Proposition~\ref{p:Tijdeman}.  
In particular, the choice $C = C' = 1 {-} \frac{1+\min_i\alpha_i}{2d-2}$ minimizes $\max(C,C')$ under the constraint $C{+}C' \ge 2 {-} \frac{1+\min\alpha_i}{d-1}$, given that  the other constraints hold when $d\geq 3$ because $1 {-} \frac{1+\min_i\alpha_i}{2d-2}\geq 1 {-} \frac{1+1/d}{2d-2}\geq 1{-}\frac1d$.
We will show that these sequences have discrepancy at most $C{+}C'$; see Proposition~\ref{prop:tijdemandiscrepancy}.

Similarly to hypercubic billiard sequences, we consider the first time that $\bx_0 {+} \mathbb{R}_+ \balpha$ hits a hyperplane with the given properties, for $\bx_0 \in \bone^\perp$ in a certain neighbourhood of~$\mathbf{0}$ that we specify later. 
We define the word $u = (u_n)_n \in  \{1,\dots,d\}^{\mathbb{N}}$ as follows. 
Given $u_{[0,n)}$, $n \ge 0$, we first consider $t_{n,i} \in \mathbb{R}$, $1 \le i \le d$, such that 
\begin{equation}\label{eq:tijde}
\bx_0 + (n {+} t_{n,i}) \balpha \in \bp(u_{[0,n)}) + C \be_i + \be_i^\bot.
\end{equation}
With $\bx_n:= \bx_0 {+} n \balpha {-} \bp(u_{[0,n)}) \in \bone^\perp$, we have 
\[
\bx_n + t_{n,i} \balpha \in C \be_i + \be_i^\bot.
\]
Writing $\bx_n = (x_{n,1},\dots,x_{n,d})$, we obtain that
\begin{equation}\label{def:t}
t_{n,i} =  \frac{C-x_{n,i}}{\alpha_i} \quad (1 \le j \le d). 
\end{equation}
We now have to choose the index~$i$ that will be the value of~$u_n$.  
We do this according to two criteria.

We first want that 
\[
\bx_{n+1} = \bx_n + \balpha - \be_{u_n} \in [-C',\infty)^d.
\]
The $i$-th coordinate of~$\bx_{n+1} $ is $x_{n,i} {+} \alpha_i {-} 1$ if $u_n=i$. 
Thus we consider only indices~$i$ such that 
\begin{equation} \label{eq:ti}
x_{n,i} + \alpha_i - 1 \ge -C' , \quad \mbox{ i.e.,  } \quad t_{n,i} \le 1 + \frac{C+C'-1}{\alpha_i}.
\end{equation}

\begin{remark}\label{rem:existence}
If $C' \ge 1 - \frac{1}{d}$, then \eqref{eq:ti} is always satisfied for some~$i$ because $x_{n,i} {+} \alpha_i {-} 1 < -C'$ for all~$i$ would imply that $\sum_{i=1}^d x_{n,i} <  d (1{-}C') {-}1 < 0$ by the assumption $C' \ge 1 {-} \frac{1}{d}$, contradicting that $\bx_n \in \bone^\perp$. 
\end{remark}

The second condition will be to take the index $i$ providing the smallest value of~$t_{n,i}$ fulfilling \eqref{eq:ti}. 
We thus set
\begin{equation} \label{eq:tni}
u_n = i \quad \mbox{if} \quad t_{n,i} = \min \{t_{n,j} \,:\, 1 \le j \le d,\, x_{n,j} {+} \alpha_j {-} 1 \ge -C'\}.
\end{equation}
When the minimum is attained for several~$i$, we choose e.g.\ the smallest~$i$. 
(Recall that, by Remark~\ref{r:choice}, the choice of $i$ in this case does not affect the set $\overline{\{\pi_{\balpha}(\bx_n) : n\in\mathbb{N}\}}$.)
We will see in Proposition~\ref{p:Tijdeman} that this yields a sequence satisfying  $\bx_n \in [-C',C]^d$ for all~$n$.

\begin{remark}\label{rem:billard}
For $C=1$ and $C' \ge (d{-}1) \|\balpha\|_\infty$, $\bx_0 \in E_{\balpha}$, we obtain the classical billiard sequences.
Indeed, $C=1$ implies similarly to Remark~\ref{rem:existence} that $\min\{t_{n,j} : 1 \le j\le d\} \le d$, thus for $j=\operatorname{argmin}\{t_{n,j} : 1 \le j\le d\}$, we have $\frac{1-x_{n,j} }{\alpha_j}\leq d$ which implies $x_{n,j}{+}\alpha_j{-}1\geq-(d{-}1)\alpha_j\geq -C'$. Therefore, the definition of $u_n$ in \eqref{eq:tni} equals the definition for billiard sequences in Section~\ref{subsec:billiard}; see in particular \eqref{eq:tnibilliard}.
For general $C \ge 0$, $C' \ge 1{-}\frac{1}{d}$, the construction ensures that $\bx_n \in [-C',\infty)^d$ for all $n \in \mathbb{N}$. 
Indeed, we have $x_{n+1,u_n} \ge -C'$ and $x_{n+1,j} = x_{n,j} {+} \alpha_j \geq x_{n,j}$ for $j \ne u_n$. 
\end{remark}

Before showing that $\bx_n \in [-C',C]^d$ for all~$n$ (see Proposition~\ref{p:Tijdeman}), we define Tijdeman's sequences in terms of dynamical systems that generate them. 
For $C \ge 0$, $C' \ge 1{-}\frac{1}{d}$, let
\begin{equation} \label{e:Si}
\begin{aligned}
S_{\balpha,C,C',i} & = \{(x_1,\dots,x_d) \in \bone^\perp \cap [-C',\infty)^d \,:\, x_i {+} \alpha_i {-} 1\ge -C'\ \mbox{and} \\
& \qquad \mbox{$\frac{C-x_i}{\alpha_i} \le \frac{C-x_j}{\alpha_j}$ for all $1 \le j \le d$ such that $x_j {+} \alpha_j {-} 1\ge -C'$}\}
\end{aligned}
\end{equation}
and set
\[
\hat{T}_{\balpha,C,C'}:\, \bone^\perp \cap [-C',\infty)^d \to \bone^\perp \cap [-C',\infty)^d , \quad \bx \mapsto \bx + \balpha - \be_i \quad \mbox{if}\ \bx \in S_{\balpha,C,C',i},
\]
where we choose the smallest such~$i$ in case that $\bx \in S_{\balpha,C,C',i} \cap S_{\balpha,C,C',j}$ for some $i \ne j$. 
Note that, for $C' \ge 1{-}\frac{1}{d}$, $\bigcup_{i=1}^d S_{\balpha,C,C',i}$ forms a topological partition of $[-C',\infty)^d \cap \bone^\perp$ and that $[-C',\infty)^d \cap \bone^\perp$ is a bounded set because it is contained in $[-C',(d{-}1)C']^d$.

\begin{definition}[Tijdeman  parameters and sequences]\label{def:tijdeman}
Let $\balpha = (\alpha_1,\dots,\alpha_d) \in (0,1)^d$ be a  totally irrational frequency vector.
Let $C \ge 0$, $C' \ge 1{-}\frac{1}{d}$, $\bx_0 \in \bone^\perp$.
If $\hat{T}_{\balpha,C,C'}^n(\bx_0) \in [-C',C]^d$ for all $n \ge 0$, then we call $(\balpha,C,C',\bx_0)$ \emph{Tijdeman parameters} and the sequence $u_0 u_1 \cdots$ such that $\hat{T}_{\balpha,C,C'}^{n+1}(\bx_0) {-} \hat{T}_{\balpha,C,C'}^n(\bx_0) = \balpha {-} \be_{u_n}$ a \emph{Tijdeman sequence}. 
\end{definition}

To obtain bounded natural partitions (as in Definition~\ref{def:coding}) for Tijdeman parameters $(\balpha,C,C',\mathbf{0})$, let
\begin{align} \label{eq:P}
D_{\balpha,C,C'} &= \overline{\{\hat{T}_{\balpha,C,C'}^n(\mathbf{0}) \,:\, n\ge 0\}}, \qquad P_{\balpha,C,C'}=\overline{\mathring{D}}_{\balpha,C,C'},\\ P_{\balpha,C,C',i} &= \overline{\mathring{P}_{\balpha,C,C'} \cap \mathring{ S}_{\balpha,C,C',i} }\qquad (1 \le i \le d) \label{eq:Pi}.
\end{align} 
Using Proposition~\ref{prop:partition}, we will see that $\{P_{\balpha,C,C',i} : 1 \le i \le d\}$ forms a natural partition of~$P_{\balpha,C,C'}$. 

Let us come back to the construction of the sequence~$u$ from~\eqref{eq:tni}.
We have $t_{n,i} = \frac{C-x_{n,i}}{\alpha_i}$ for all $n \in \mathbb{N}$, $1 \le i \le d$. 
Therefore, when $\bx_n\in S_{\balpha,C,C',i}$ we have $u_n = i$, i.e., $\bx_n \in S_{\balpha,C,C',u_n}$, and $\hat{T}_{\balpha,C,C'}(\bx_n) = \bx_{n+1}$, which implies that $\bx_n = \hat{T}_{\balpha,C,C'}^n(\bx_0)$, $n \in \mathbb{N}$, and $u$ is the coding of~$\bx_0$ w.r.t.\ the partition $\{S_{\balpha,C,C',i} : 1 \le i \le d\}$. 
Note that, for $\bx_0 = \mathbf{0}$, $P_{\balpha,C,C',i}$ is the closure of the set of points $n\balpha {-} \bp(u_{[0,n)})$ with $u_n = i$.

We will see with Proposition~\ref{p:1} below that $P_{\balpha,C,C'}$ is a measurable fundamental domain of $\bone^\perp / (\mathbb{Z}^d \cap \bone^\perp)$ admitting the partition $P_{\balpha,C,C'} = \bigcup_{i=1}^d P_{\balpha,C,C',i}$. 
But first let us state the following key  proposition, which provides Tijdeman parameters (see Definition~\ref{def:tijdeman}) and is inspired by \cite[Theorem~1]{Tijdeman:80}.

\begin{proposition} \label{p:Tijdeman} 
Let $\balpha = (\alpha_1,\dots,\alpha_d) \in (0,1)^d$ be a totally irrational frequency vector. Let $C,C' $ satisfy (\ref{e:CC}), i.e., $C,C' \ge 1{-}\frac{1}{d}$, $C \le 1$, and $C{+}C' \ge 2{-}\frac{1+\min_i \alpha_i}{d-1}$. 
Let $\hat{T}_{\balpha,C,C'}$ be the map from Definition~\ref{def:tijdeman} and $\bx_0 \in [C{-}1,C]^d \cap \bone^\perp$. 
Then $\hat{T}_{\balpha,C,C'}^n(\bx_0) \in [-C',C]^d$ for all $n \in \mathbb{N}$. 
\end{proposition}

\begin{proof}
Write $\hat{T}_{\balpha,C,C'}^n(\bx_0) = \bx_n = (x_{n,1},\dots,x_{n,d})$ and let $u_n$ be such that $\bx_{n+1} = \bx_n {+} \balpha {-} \be_{u_n}$. 
Suppose that $\bx_n \notin [-C',C]^d$ for some $n \ge 1$, and let $n$ be minimal with this property. 

We have seen above that $C' \ge 1 {-} \frac{1}{d}$ implies that $\bx_n \in [-C',\infty)^d$, thus there exists~$i$ such that $x_{n,i} > C$. 
Then $t_{n,i}<0$ and $x_{n,i}{+}\alpha_i{-}1 > C{-}1 \geq -C'$, hence also
 $t_{n,u_n} < 0$, i.e., 
\begin{equation} \label{e:xnvn}
x_{n,u_n} > C.
\end{equation}
Then the set 
\[
W = \big\{i \in \{1,\dots,d\} \,:\, x_{n,i} < C{-}1\big\}
\]
is not empty since otherwise we would have 
\[
\sum_{i=1}^d  x_{n,i} > C+ (d{-}1) (C{-}1) \geq 0
\]
(using that $C \ge 1{-}\frac{1}{d}$), contradicting that $\bx_n \in \bone^\perp$.

Let $m < n$ be maximal such that  $u_m \in W$. 
Such an $m \ge 0$ exists because otherwise we have $x_{n,i} = x_{0,i} + n \alpha_i$ for all $i \in W$, giving the contradiction $C{-}1 \le x_{0,i} < x_{n,i} < C{-}1$.
Let 
\[
V = \{u_k \,:\, m < k < n\}.
\]
Using $\bx_m + (n{-}m) \balpha = \bx_n + \bp(u_{[m,n)})$, we obtain that
\begin{align}
& x_{m,u_m} + (n{-}m) \alpha_{u_m} = x_{n,u_m} + 1 < C & \hspace{-3em} \mbox{(since $|u_{[m,n)}|_{u_m} = 1$ and $u_m \in W$)}, \label{e:W1} \\
& x_{m,u_n} + (n{-}m) \alpha_{u_n} \ge x_{n,u_n} > C & \hspace{-3em} \mbox{(since $|u_{[m,n)}|_{u_n} \ge 0$ and \eqref{e:xnvn} holds)}, \label{e:W2} \\
& x_{m,i} + (n{-}m) \alpha_i \ge x_{n,i} + 1 \ge C \ \mbox{for all}\ i \in V & \mbox{(since $|u_{[m,n)}|_i \ge 1$ and $i \notin W$)}. \label{e:W3}
\end{align}
For all $i \in V \cup \{u_n\}$, we have that $x_{m,i}{+}\alpha_i{-}1<-C'$.
Indeed, we have $t_{m,i} \le n-m < t_{m,u_m}$ by \eqref{def:t} and
\eqref{e:W1}--\eqref{e:W3}, hence $x_{m,i}{+}\alpha_i{-}1 \ge -C'$ would
contradict the definition of~$u_m$ in \eqref{eq:tni}.
This implies that $x_{m+1,i} < 1{-}C'$.
We obtain
\[
x_{n,u_n} - x_{m+1,u_n} > C+C'-1, \qquad x_{n,i} - x_{m+1,i} > C+C'-2 \quad \mbox{for all}\ i \in V.
\]
In particular, we have $n \neq m+1$ since $C+C' \ge 1$, thus 
\[
x_{n,i} - x_{m+1,i} = (n{-}m{-}1) \alpha_i \ge \alpha_i \quad \mbox{for all}\ i \notin V.
\]
Using that $u_m \notin V \cup \{u_n\}$, in particular that $V \setminus \{u_n\}$ has at most $d{-}2$ elements, and distinguishing between the cases $C+C' \ge 2$ and $C+C' < 2$, we obtain that 
\begin{align*}
0 & = \sum_{i=1}^d x_{n,i} - \sum_{i=1}^d x_{m+1,i} \ge x_{n,u_n} - x_{m+1,u_n} + x_{n,u_m} - x_{m+1,u_m} + \sum_{i \in V \setminus \{u_n\}} (x_{n,i} - x_{m+1,i}) \\
& > C+C'-1 + \alpha_{u_m} + \min\big\{0, (d{-}2) (C{+}C'{-}2)\big\} \\
& \ge \min\{C+C'-1, (d{-}1)(C{+}C'{-}2) + 1 +\min_i \alpha_i\},
\end{align*}
which contradicts the assumptions $C {+} C' \ge 2\, (1{-}\frac{1}{d}) \ge 1$ and $C{+}C' \ge 2{-}\frac{1+\min_i \alpha_i}{d-1}$.
This implies that $\bx_n \in [-C',C]^d$ for all $n \in \mathbb{N}$.
\end{proof}

We can now conclude with the following two propositions.

\begin{proposition} \label{p:1}
	Let $(\balpha,C,C',\mathbf{0})$ be Tijdeman parameters with $C,C' \in \big[1{-}\frac{1+\min\alpha_i}{2d-2}, 1\big)$.
	Let $\hat{T}_{\balpha,C,C'}$ be the map from Definition~\ref{def:tijdeman} and let $P_{\balpha,C,C'}$ be defined as in~\eqref{eq:P}.
	Then $P_{\balpha,C,C'}$ is a measurable fundamental domain of $\bone^\perp / (\mathbb{Z}^d \cap \bone^\perp)$ admitting the natural partition $P_{\balpha,C,C'} = \bigcup_{i=1}^d P_{\balpha,C,C',i}$ defined as in ~\eqref{eq:Pi}. 
	The restriction of $\hat{T}_{\balpha,C,C'}$ to~$P_{\balpha,C,C'}$ is measurably conjugate to the translation~$T_{\balpha}$ on~$\mathbb{T}^{d-1}$.  
	Moreover, the $P_{\balpha,C,C',i}$ are finite unions of   convex polytopes with nonempty interiors, and $(\balpha,C,C',\bx_0)$ are Tijdeman parameters for almost all $\bx_0 \in P_{\balpha,C,C'}$.
\end{proposition}

\begin{proof}
We want to use Proposition~\ref{prop:partition} with $V = \bone^\perp$, $\LL=V\cap \Z^d$, $R = V \cap [-C',\infty)^d$, $R_i=\overline{S}_i$ with $S_i= S_{\balpha,C,C',i}$ defined in \eqref{e:Si}, $1 \le i \le d$, $\widehat{T} = \hat{T}_{\balpha,C,C'}$, $T_i(\bx)=\bx{+}\balpha{-}\be_i$, $K = [-C',C]^d \cap V$, $\bx_0=0$ and $U = (C{-}1,1{-}C')^d \cap V$. If we can, then we are done, except that Proposition~\ref{prop:partition} gives  $P_{\balpha,C,C',i}=\overline{\mathring P\cap \mathring{\overline{S}_i}}$ instead of $P_{\balpha,C,C',i}=\overline{\mathring P\cap \mathring{S_i}}$ as in \eqref{eq:Pi}. So we have to check that Assumptions (\ref{i:a1})--(\ref{i:a5}) from Section~\ref{sec:constr-natur-part} hold and that $\mathring{\overline{S}_i}=\mathring{ S_i}$. It is clear that Assumption~(\ref{i:a1}) about minimality and Assumption~(\ref{i:a2}), i.e., $R_i$ is closed and $T_i(R_i)\subset R$,   hold.  
	
	Let us prove that Assumption~(\ref{i:a4}) holds with $K$, $\bx_0$ and $U$ as above. On the one hand, by Proposition~\ref{p:Tijdeman}, $\{\hat{T}_{\balpha,C,C'}^n(0):n\in\N\} \subset K$  when $C,C' \ge 1{-}\frac{1}{d}$, $C \le 1$, and $C{+}C' \ge 2{-}\frac{1+\min_i \alpha_i}{d-1}$. On the other hand, $U \neq \emptyset$ and $K\cap(U+\bn)=\emptyset$ for all $\bn\in\Z\setminus\{0\}$, when $C,C'<1$. Hence, Assumption~(\ref{i:a4}) holds when  $C,C' \in \big[1{-}\frac{1+\min\alpha_i}{2d-2}, 1\big)$. 
	
We now prove Assumption~(\ref{i:a3}), i.e., $\operatorname{Leb}(R_i\cap R_j)=0$ for all $i\neq j$, and Assumption~\eqref{i:a5}, i.e., each $R_i$  is a finite union of convex polytopes with nonempty interiors. Consider the sets $Q_i \subset S_i$, $i=1,\dots,d$, defined as
\[
Q_i = \{(x_1,\dots,x_d) \in R : x_i {+} \alpha_i {-} 1\ge -C',\, \mbox{$\frac{C-x_i}{\alpha_i} < \frac{C-x_j}{\alpha_j}$ for all $j \ne i$ s.t.\ $x_j {+} \alpha_j {-} 1\ge -C'$}\}.
\]
Clearly the  sets $ Q_1 , \dots,Q_d$ are disjoint. The sets $Q_i$ can be decomposed into convex~subsets:
\[
Q_i = \bigcup_{J\subset\{1,\dots,d\} ,\,  i\in J} Q_{i,J},
\]
where
\[
\begin{aligned}
Q_{i,J} = \{(x_1,\dots,x_d) \in R \,:\, & x_j {+} \alpha_j {-} 1 \ge -C'\ \mbox{if and only if}\ j\in J,\\
& \mbox{$\frac{C-x_i}{\alpha_i} < \frac{C-x_j}{\alpha_j}$ for all $j \in J \setminus \{i\}$}\}.
\end{aligned}
\]
Clearly, each~$Q_{i,J}$ is a finite intersection of half spaces, open or closed.  Hence, each~$\mathring{Q}_{i,J}$ is a finite intersection of half open spaces.  
Thus the sets  
\[
Q'_i = \bigcup_{i\in J\subset\{1,\dots,d\}:\,\mathring{Q}_{i,J}\neq\emptyset} \overline{Q_{i,J}} 
\]
are finite unions of convex polytopes with nonempty interiors that are finite intersections of open half spaces.
If we can show that, for each $i$, $S_i\subset Q'_i$, then we are done. 
Indeed, on the one hand $Q'_i\subset \overline Q_i\subset \overline{S}_i=R_i$ so that $R_i=Q'_i$ and by definition, $Q'_i$ is a finite union of convex polytopes with nonempty interiors, which is Assumption \eqref{i:a5}.   On the other hand, $R_i\setminus Q_i=Q'_i\setminus Q_i$ is negligible because it is included in the union of the boundaries $\partial Q_{i,J}$, and if $i\neq j$, $R_i\cap R_j\subset (R_i\setminus Q_i)\cup(R_j\setminus Q_j)$ because $Q_i\cap Q_j=\emptyset$. Therefore, Assumption~\eqref{i:a3} holds.

\smallskip
In order to show that $S_i\subset Q'_i$, we show first that $\mathring{S}_i\subset Q_i$, then $\mathring{S}_i\subset Q'_i$ and  $S_i\subset \overline{\mathring{S}}_i$, because this leads to the inclusions $S_i\subset  \overline{\mathring{S}}_i\subset\overline{Q'_i}=Q'_i$.
By contradiction, suppose that there exists $\bx\in\mathring{S}_i\setminus Q_i$. By definition of $S_i$ and $Q_i$, there exists $j\ne i$ such that $x_j{+}\alpha_j{-}1\geq -C'$ and $\frac{C-x_i}{\alpha_i} =\frac{C-x_j}{\alpha_j}$. Moreover, since $\bx\in\mathring{S}_i$, for $t>0$ small enough, $\by_t=\bx{+}t(\be_j{-}\be_i)$ is in~$S_i$. But the coordinates $y_{t,k}$ of~$\by_t$ satisfy $y_{t,j}{+}\alpha_j{-}1\geq x_j{+}\alpha_j{-}1\geq -C'$ and $\frac{C-y_{t,i}}{\alpha_i} = \frac{C-x_i+t}{\alpha_i} > \frac{C-x_j-t}{\alpha_j} = \frac{C-y_{t,j}}{\alpha_i}$, a contradiction.

\smallskip
Let us show that  $\mathring{S}_i\subset Q'_i$.
Observe first, that if for some pair~$(i,J)$, $\mathring{Q}_{i,J} = \emptyset$, then $Q_{i,J}$ is contained in a hyperplane of~$V$ and therefore $\mathring{\overline{Q}}_{i,J} = \emptyset$.
By contradiction, suppose that  there exists  $\bx\in\mathring{S}_i\setminus Q'_i$. Since $Q'_i$ is closed, there exists $r>0$ such that $B(\bx,r)\subset \mathring{S}_i\setminus Q'_i$.  Now  $\mathring{S}_i\subset Q_i$, so that  $B(\bx,r)\subset Q_i\setminus Q'_i$ which is impossible for the $Q_{i,J}$ that are missing in $Q'_i$ have empty interiors. 
	
\smallskip	
Finally, let us show that $S_i\subset \overline{\mathring{S}}_i$.
Let $\bx\in S_i$, let $I(\bx)=\{j\ne i:x_j{+}\alpha_j{-}1\geq -C'\}$ and $J(\bx) = \{j \ne i : x_j=-C'\}$. Suppose $I(\bx)\neq \emptyset$. For all $j\in I(\bx)$, $\frac{C-x_j}{\alpha_j}\geq \frac{C-x_i}{\alpha_i}$. 
Let 
\[
\by_t = (y_{t,1},\dots,y_{t,d}) = \bx + t \bigg(\card I(\bx) \bigg(\be_i{+}\sum_{j\in J(\bx)} \be_j\bigg) - (\card J(\bx){+}1) \sum_{j\in I(\bx)} \be_j\bigg).
\] 
For $t>0$ small enough, $I(\by_t)\subset I(\bx)$ and for all $j\in I(\bx)$, one has
\[
\frac{C-y_{t,j}}{\alpha_j} > \frac{C-x_j}{\alpha_j} \geq \frac{C-x_i}{\alpha_i} > \frac{C-x_i-t\card I(\bx)}{\alpha_i}.
\]
Therefore, $\by_t$ is in~$S_i$ for all $t>0$ small enough, and $\by_t$ is even in $\mathring{S}_i$ for all $t>0$ small enough because there are only finitely many $t>0$ such that  $y_{t,j}{+}\alpha_j{-}1= -C'$ for some $j\in I(\bx)$ and all the inequalities are strict. If $I(\bx) = \emptyset$ and $x_i{+}\alpha_i{-}1=-C'$, then there exists $k\ne i$ such that $x_k>-C'$ because $\sum_j x_j = 0$. Then we use $\by_t = \bx{+}t(\sum_{j\neq k}\be_j{-}(d{-}1)\be_k)$ with $t>0$ small enough to conclude as before. If $I(\bx)=\emptyset$ and $x_i{+}\alpha_i{-}1>-C'$, then we use instead $\by_t = \bx{+}t(\sum_{j \ne i}\be_k{-}(d{-}1)\be_i)$ with $t>0$ small enough. 

\smallskip
It remains  to check that $\mathring{\overline{S}}_i=\mathring{S}_i$ to prove that~\eqref{eq:Pi} holds. It suffices to prove that  $\mathring{\overline{S}}_i\subset S_i$. By contradiction, if $\bx\in \mathring{\overline{S}}_i\setminus S_i$, there exists $j\neq i$ such that $\bx\in S_j$ and there exists $r>0$ such that $B(\bx,r)\subset \overline{S}_i$.  Now,  $\overline{S}_i=\overline{\mathring{S}_i}$ and $\overline{S}_j=\overline{\mathring{S}_j}$, so that $\mathring{S}_i$ is dense in $B(\bx,r)$ and $B(\bx,r)\cap \mathring{S}_j\neq \emptyset$. Thus $\mathring{S}_i\cap\mathring{S}_j\neq \emptyset$, a contradiction.
\end{proof}

The following  result on the discrepancy of Tijdeman sequences can be considered as a counterpart of Proposition~\ref{prop:eqbillard}.
It provides suitable parameters that yield fairly distributed sequences.

\begin{proposition}\label{prop:tijdemandiscrepancy}
Let $u$ be a Tijdeman sequence with parameters $(\balpha,C,C',\bx_0)$, $C,C' \in \big[1{-}\frac{1+\min\alpha_i}{2d-2}, 1\big)$.
Then the sequence~$u$ is eventually a natural coding of the dynamical system $(P_{\balpha,C,C'}, \hat{T}_{\balpha,C,C'})$.
Moreover,  for the dynamical system $(P_{\balpha,C,C'}, \hat{T}_{\balpha,C,C'})$ and for any $\bx_0 \in P_{\balpha,C,C'}$, the discrepancy of the coding~$u$ of~$\bx_0$ satisfies
\[
\frac{\operatorname{diam}P_{\balpha,C,C'}}{2} \le \Delta_{\balpha}(u) \leq \operatorname{diam}P_{\balpha,C,C'}\le C+C',
\]
and, for $\bx_0 = \mathbf{0}$, the discrepancy satisfies 
\[
\Delta_{\balpha}(u) \le \max \{C,C'\}.
\]
Hence, if $C = C '= 1 {-} \frac{1+\min_i\alpha_i}{2d-2}$ and $\bx_0 = \mathbf{0}$, then 
\[
\Delta_{\balpha}(u)  \leq  1 - \frac{1+\min_i\alpha_i}{2d-2} < 1-\frac{1}{2d-2} = D_d.
\]
\end{proposition}
\begin{proof}

If we start from some $\bx_0 \notin P_{\balpha,C,C'}$, then by minimality  $\hat{T}_{\balpha,C,C'}^n(\bx_0) \in (C{-}1, 1{-}C')^d \cap \bone^\perp \subset P_{\balpha,C,C'}$ for some $n \ge 1$. Next, by Lemma~\ref{lem:hyperplane}, for all~$n$ large enough, $\hat{T}_{\balpha,C,C'}(\bx_0)$ is in none of the boundaries of the~$P_{\balpha,C,C',i} $ because these boundaries are included in a finite union of subspaces of dimension~$d{-}2$. Therefore, the sequence~$u$ is eventually a natural coding of the dynamical system $(P_{\balpha,C,C'}, \hat{T}_{\balpha,C,C'})$.
We conclude about the discrepancy as in (\ref{eq:eqbillard})  in the proof of  Proposition~\ref{prop:eqbillard}.
\end{proof}

Figure~\ref{f:2} shows a case where the sets~$P_{\balpha,C,C'}$ are polygons that are more complicated than the sets~$E_{\balpha,i}$ (which correspond to the 
case of the  hypercubic billiard sequences such as developed in Section~\ref{subsec:billiard}).
For large~$C'$, we would be in the (shifted) billiard case, with $P_{\balpha,C,C'} = E_{\balpha,i} {-} (1{-}C) \pi_{\balpha}(\bone)$.
However, for the parameters of Figure~\ref{f:2}, we do not have $E_{\balpha} {-} (1{-}C) \pi_{\balpha}(\bone) \subset [-C',C]^d$.
Here, the leftmost part of $E_{\balpha,1} - (1{-}C) \pi_{\balpha}(\bone)$ is not in $S_{\balpha,C,C',1}$ but in $S_{\balpha,C,C',2} \cup S_{\balpha,C,C',3}$.
(The figure depicts the $\iota$-images of the sets.)
  
\begin{figure}[ht] 
\begin{tikzpicture}[scale=4]
\newcommand{\alphaone}{.5}
\newcommand{\alphatwo}{.3}
\newcommand{\alphathree}{.2}
\newcommand{\toright}{1.85}
\newcommand{\Pone}{\filldraw[fill=blue!10](-.25,.15)--(-.25,-.1)--(.125,-.625)--(.25,-.7)--(.25,-.55)--(.625,-.325)--(.25,.2)--(.25,.3)--(.125,.375)--cycle; \draw[dotted](.125,-.625)--(.25,-.55) (.125,.375)--(.25,.2);}
\newcommand{\Ptwo}{\filldraw[fill=red!10](-.25,0)--(-.25,.15)--(.125,.375)--(-.375,.675)--(-.75,.45)--(-.75,.3)--cycle; \draw[dotted](-.375,.075)--(-.25,.15);}
\newcommand{\Pthree}{\filldraw[fill=green!10](-.25,0)--(-.25,-.1)--(.125,-.625)--(-.375,-.325)--(-.75,.2)--(-.75,.3)--cycle; \draw[dotted](-.375,.075)--(-.25,-.1);}
\newcommand{\coord}{\draw[->](-.85,0)--(.85,0); \draw[->](0,-.8)--(0,.85); \draw(-.75,-.02)node[below]{${-}C'$}--(-.75,.02) (.75,-.02)node[below]{$C$}--(.75,.02) (-.02,-.75)node[left]{${-}C'$}--(.02,-.75) (-.02,.75)node[left]{$C$}--(.02,.75); \draw[dotted](-.75,0)--(0,-.75)--(.75,-.75)--(.75,0)--(0,.75)--(-.75,.75)--cycle;}
\Pone \node at (.125,-.125){$P_1$};
\begin{scope}[shift={(\toright+\alphaone-1,\alphatwo)}]
\Pone \node at (.125,-.125){$P_1{+}\balpha{-}\be_1$};
\end{scope}
\Ptwo \node at (-.375,.375){$P_2$};
\begin{scope}[shift={(\toright+\alphaone,\alphatwo-1)}]
\Ptwo \node at (-.375,.375){$P_2{+}\balpha{-}\be_2$};
\end{scope}
\Pthree \node at (-.375,-.125){$P_3$};
\begin{scope}[shift={(\toright+\alphaone,\alphatwo)}]
\Pthree \node at (-.375,-.125){$P_3{+}\balpha{-}\be_3$};
\end{scope}
\coord
\begin{scope}[shift={(\toright,0)}]
\coord
\end{scope}
\end{tikzpicture}
\caption{The polygons $P_i = P_{\balpha,C,C',i}$ for $\balpha \approx (0.5,0.3,0.2)$, $C=C'=3/4$.} \label{f:2}
\end{figure}
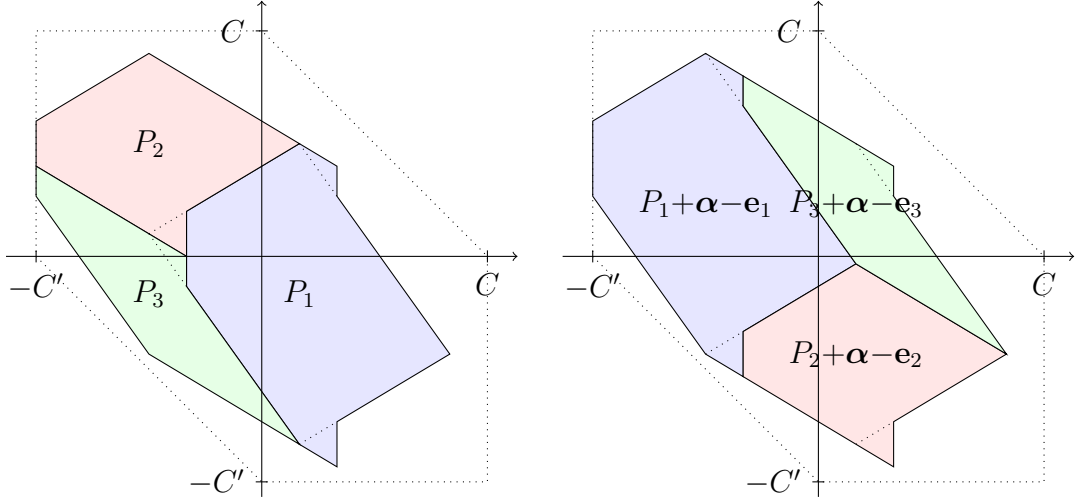

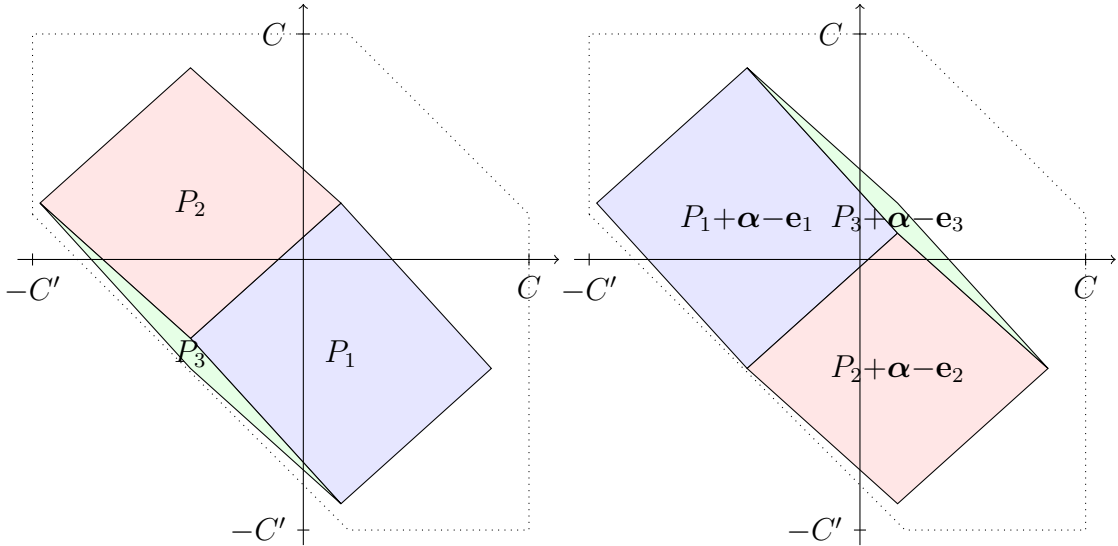
\begin{figure}[ht] 
\begin{tikzpicture}[scale=4]
\newcommand{\alphaone}{.5}
\newcommand{\alphatwo}{.45}
\newcommand{\alphathree}{.05}
\newcommand{\toright}{1.85}
\newcommand{\Pone}{\filldraw[fill=blue!10](-.375,-.2625)--(.125,-.8125)--(.625,-.3625)--(.125,.1875)--cycle;}
\newcommand{\Ptwo}{\filldraw[fill=red!10](-.375,-.2625)--(.125,.1875)--(-.375,.6375)--(-.875,.1875)--cycle;}
\newcommand{\Pthree}{\filldraw[fill=green!10](-.375,-.2625)--(.125,-.8125)--(-.375,-.3625)--(-.875,.1875)--cycle;} 
\newcommand{\coord}{\draw[->](-.95,0)--(.85,0); \draw[->](0,-.95)--(0,.85); \draw(-.9,-.02)node[below]{${-}C'$}--(-.9,.02) (.75,-.02)node[below]{$C$}--(.75,.02) (-.02,-.9)node[left]{${-}C'$}--(.02,-.9) (-.02,.75)node[left]{$C$}--(.02,.75); \draw[dotted](-.9,.15)--(.15,-.9)--(.75,-.9)--(.75,.15)--(.15,.75)--(-.9,.75)--cycle;}
\Pone \node at (.125,-.3125){${P}_1$};
\begin{scope}[shift={(\toright+\alphaone-1,\alphatwo)}]
\Pone \node at (.125,-.3125){${P}_1{+}\balpha{-}\be_1$};
\end{scope}
\Ptwo \node at (-.375,.1875){${P}_2$};
\begin{scope}[shift={(\toright+\alphaone,\alphatwo-1)}]
\Ptwo \node at (-.375,.1875){${P}_2{+}\balpha{-}\be_2$};
\end{scope}
\Pthree \node at (-.375,-.3125){${P}_3$};
\begin{scope}[shift={(\toright+\alphaone,\alphatwo)}]
\Pthree \node at (-.375,-.3125){${P}_3{+}\balpha{-}\be_3$};
\end{scope}
\coord
\begin{scope}[shift={(\toright,0)}]
\coord
\end{scope}
\end{tikzpicture}
\caption{The parallelograms $P_i = P_{\balpha,C,C',i}$ for $d=3$, $\balpha \approx (0.5,0.45,0.05)$, $C=3/4$, $C'=9/10$.} \label{f:3ter}
\end{figure}

\begin{figure}[ht] 
\begin{tikzpicture}[scale=4]
\newcommand{\alphaone}{.5}
\newcommand{\alphatwo}{.45}
\newcommand{\alphathree}{.05}
\newcommand{\toright}{1.85}
\newcommand{\Pone}{\filldraw[fill=blue!10](-.25,-.15)--(-.25,-.4)--(.065,-.75)--(.25,-.75)--(.25,-.7)--(.69444,-.3)--(.56875,-.3)--(.25,.05)--(.25,.25)--(.19444,.25)--cycle;
\draw[dotted](.2,-.75)--(.25,-.7) (.25,.05)--(.125,.1875)--(.25,.075) 
(.625,-.3625)--(.56875,-.3);} 
\newcommand{\Ptwo}{\filldraw[fill=red!10](-.44444,-.2)--(-.25,-.2)--(-.25,-.15)--(.19444,.25)--(.05556,.25)--(-.25,.525)--(-.25,.7)--(-.30556,.7)--(-.75,.3)--(-.75,.075)--cycle; \draw[dotted](-.30556,-.2)--(-.25,-.15) (.125,.1875)--(.05556,.25) (-.375,.6375)--(-.25,.525);}
\newcommand{\Pthree}{\filldraw[fill=green!10](-.44444,-.2)--(-.25,-.2)--(-.25,-.3)--(-.435,-.3)--(-.75,.05)--(-.75,.075)--cycle (-.25,-.4)--(.065,-.75)--(.05556,-.75)--(-.25,-.475)--cycle; \draw[dotted](-.375,-.2625)--(-.34125,-.3) (-.44444,-.2)--(-.33333,-.3);} 
\newcommand{\coord}{\draw[->](-.85,0)--(.85,0); \draw[->](0,-.85)--(0,.85); \draw(-.75,-.02)node[below]{${-}C'$}--(-.75,.02) (.75,-.02)node[below]{$C$}--(.75,.02) (-.02,-.75)node[left]{${-}C'$}--(.02,-.75) (-.02,.75)node[left]{$C$}--(.02,.75); \draw[dotted](-.75,0)--(0,-.75)--(.75,-.75)--(.75,0)--(0,.75)--(-.75,.75)--cycle;}
\Pone \node at (.125,-.3125){$P_1$};
\begin{scope}[shift={(\toright+\alphaone-1,\alphatwo)}]
\Pone \node at (.125,-.3125){$P_1{+}\balpha{-}\be_1$};
\end{scope}
\Ptwo \node at (-.375,.1875){$P_2$};
\begin{scope}[shift={(\toright+\alphaone,\alphatwo-1)}]
\Ptwo \node at (-.375,.1875){$P_2{+}\balpha{-}\be_2$};
\end{scope}
\Pthree \node at (-.35,-.4){$P_3$};
\begin{scope}[shift={(\toright+\alphaone,\alphatwo)}]
\Pthree \node at (0,-.25){$P_3{+}\balpha{-}\be_3$};
\end{scope}
\coord
\begin{scope}[shift={(\toright,0)}]
\coord
\end{scope}
\end{tikzpicture}
\caption{For $d=3$, $\balpha \approx (0.5,0.45,0.05)$, $C=C'=3/4$, the $P_i= P_{\balpha,C,C',i}$ are unions of polygons.} \label{f:3bis}
\end{figure}
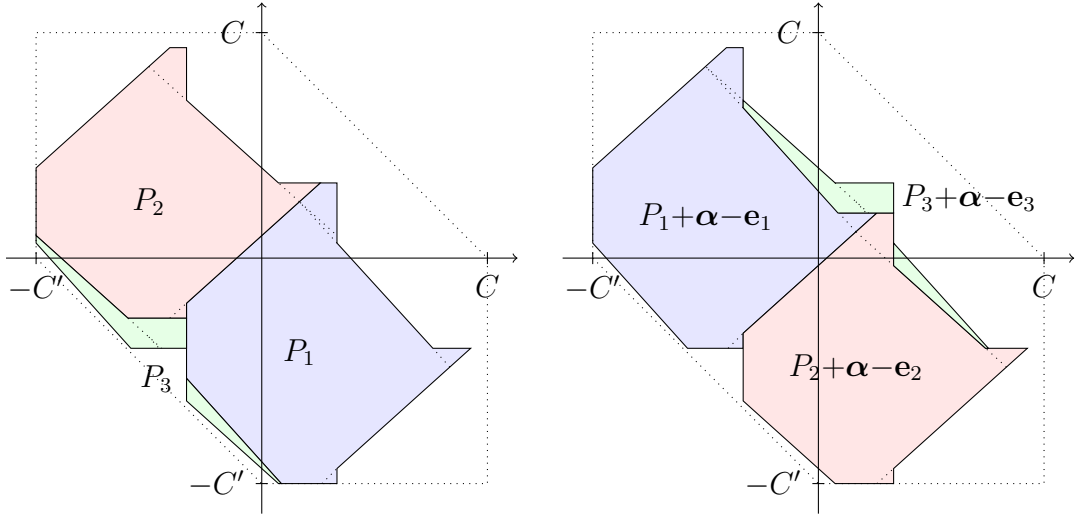

\section{Proof of Theorem \ref{theo:main}}\label{sec:proof}

 Let $u$ be a generalized  Tijdeman sequence with parameters $(\balpha, C,C', \bx_0)$  where $\balpha = (\alpha_i)_{1 \leq i \leq d}$ is a totally irrational frequency vector. Then $u$ is eventually the bounded natural coding of~$T_{\balpha}$ via  a partition of a  measurable fundamental domain of~$\mathbb{T}^{d-1}$ into $d$ finite unions of  convex polytopes, according to Definition~\ref{def:nc} and Propositions~\ref{p:1} and~\ref{prop:tijdemandiscrepancy}.
The fact that $u$ is a bounded natural coding implies the dynamical properties stated in Theorem~\ref{theo:main}. Indeed, the shift $(X,S)$ is minimal, uniquely ergodic, and has purely discrete spectrum according to \cite[Theorems~A and~B]{Chevallier}  (for $d=3$) and \cite[Lemma 5.12]{BST:20} for $d>3$.

It remains to prove the statement on the factor complexity.
Note that such a statement is now classical in the setting of cut and project schemes with a polytopal window; see e.g.\ \cite{Chevallier,Jul.10,KW:19}. 
More precisely, convex polytopal windows are considered in \cite[Theorem~7.1]{KW:19} and \cite[Theorem~2.2]{KW:22}; see also the papers \cite{Chevallier,Jul.10} dealing with dimension $d=3$.
The case of a finite union of convex polytopes in arbitrary dimension, as in our setting, has been dealt with recently in \cite{Walton:24}, and the statement on the factor complexity in Theorem~1.1 can be derived from \cite[Theorem~2.6]{Walton:24} (see also \cite[Remark~1.5]{Walton:24}).
For completeness, we provide a self-contained proof here, and we establish the lower bound under a slightly more general assumption than in \cite{Walton:24}.
 Specifically, we only assume that the boundaries of the partition atoms include portions of ``linearly independent'' hyperplanes with nonempty relative interiors; see Theorem~\ref{thm:complexity}.

The proof runs from Section~\ref{subsec:idea} to~\ref{subsec:fi}.
The general strategy is described in Section~\ref{subsec:idea}.
The  upper bound is provided in Section~\ref{subsec:upper} (see Proposition~\ref{prop:complexity}), and the lower bound is based  on  Theorem~\ref{thm:complexity} (see  Section~\ref{subsec:lower}). 
The proof of the lower bound  lastly  requires some assumption stated in   Theorem~\ref{thm:complexity}, which is handled in Section~\ref{subsec:fi}.

\subsection{General strategy  for estimating the factor complexity} \label{subsec:idea}
In this and the two subsequent sections, we work with maps acting on the $d$-dimensional torus for the sake of simplicity in the notation, although we have considered so far the $(d{-}1)$-dimensional torus when handling an alphabet of size~$d$. 

Let $T_{\balpha}$ be the translation defined by $T_{\balpha}: \bx \in\T^d \mapsto \bx{+}\ttt \in \T^d$ assumed to be minimal, i.e., the coordinates of $(1, \balpha)$ are rationally independent. Suppose that 
\[
\T^d = \WW_1 \cup \WW_2 \cup \cdots \cup \WW_k,
\]
where the sets~$\WW_i$, $i=1,\dots,k$, are closed sets with disjoint interiors.  
We want to estimate the complexity of the coding of a trajectory $(T_{\balpha}^n (\bx))_n$ with respect to the partition $\{\WW_1, \WW_2 , \dots, \WW_k\}$.

Let  $F = \bigcup_{i=1}^k \partial W_i$ denote the union of the boundaries of the~$W_i$. For a positive integer~$n$, let $\mathcal{F}_n$ denote the set of connected components of $\mathbb{T}^d \setminus \bigcup_{i=0}^n T_{\balpha}^{-i}(F)$.

The main observation is that if $\bx$ and~$\by$ are in the same connected component $U \in \mathcal{F}_n$ then, for $j=1,\dots,n$, the two points $T_{\balpha}^j(\bx)$ and $T_{\balpha}^j(\by)$ are in the same~$W_i$ and even in the interior of the same~$W_i$.
Indeed, let $\gamma: [0,1] \fff U$ be a continuous path going from $\bx$ to~$\by$. If for some $j \in \{0,\dots,n\}$, $T_{\balpha}^j(\bx) \in W_i$ and $T_{\balpha}^j(\by) \not\in W_i$, then the path $T_{\balpha}^j \circ\gamma$ would have to intersect the boundary of~$W_i$, but this impossible for, by definition of~$F$, $U$~does not meet $T_{\balpha}^{-j}(\partial W_i)$. 

It follows that if $u = (u_n)_{n\in\N}$ is the coding sequence of a point $\bx_0 \in \T^d$ such that, for all $n \in \N$, $T^n(\bx_0)\in \mathring\WW_{u_n}$, then the  factor complexity of the sequence $u=(u_n)_{n\in\N}$ satisfies  
\begin{equation}\label{e:complexity}
p_u(n{+}1) \leq \card\mathcal{F}_n \quad \mbox{for all}\ n \in \N.
\end{equation}
When the sets~ $W_i$ are finite unions of convex polytopes, using translates of the hyperplanes supporting the facets of the convex polytopes,  it is easy to bound the above $\card \mathcal{F}_n$; this is the object of Proposition~\ref{prop:complexity} below.

\subsection{Upper bound for the factor complexity} \label{subsec:upper}
We will need the following lemma, whose proof is a double induction on the dimension and~$n$, along with a use of the formula $\binom{n}{p} {+} \binom{n}{p-1} = \binom{n+1}{p}$.

\begin{lemma}\label{lem:big-O} \cite[Theorem 14.2.5]{Boissonat-Yvinec}
If $H_1,\dots,H_n$ are $n$ affine hyperplanes in~$\mathbb{R}^d$, then $\mathbb{R}^d \setminus (H_1\cup\dots\cup H_n)$ has at most $\binom{n}{0} {+} \binom{n}{1} {+} \cdots {+}\binom{n}{d}$ connected components. Hence the number of connected components of the complement of the union of~$n$ hyperplanes in~$\mathbb{R}^d$ is~$O(n^d)$.
\end{lemma} 
	
Call $\pt: \mathbb{R}^d \fff \T^d$ the projection $\bx \in \mathbb{R}^d \mapsto  \bx{+}\Z^d \in \T^d$. 
Fix a norm on~$\mathbb{R}^d$ and denote by $B(\bx,r)$ the open ball of center $\bx\in\mathbb{R}^d$ and radius~$r$ associated with this norm.

\begin{proposition}\label{prop:complexity}
Let $\ttt \in \mathbb{R}^d$ and let $T_{\balpha}: \T^d \fff \T^d$ be the translation defined by $T_{\balpha}(\bx) = \bx{+}\ttt \in \T^d$. 
Suppose that 
\[
\T^d = \WW_1 \cup \WW_2 \cup \dots \cup \WW_k,
\]
where the sets~$\WW_i$, $i=1,\dots,k$ have disjoint interiors and are the projections  of finite unions of  (bounded)  convex polytopes. 

Then there exists a constant $C>0$ such that, when $u = (u_n)_{n\in\N}$ is the coding sequence of a point $\bx_0 \in \T^d$ satisfying $T^n(\bx_0)\in \mathring\WW_{u_n}$ for all $n\in\N$, the factor complexity of~$u$ satisfies  
\[
p_u(n) \leq Cn^d \quad \mbox{for all}\ n \in \N.
\]
\end{proposition}

\begin{proof}
We follow the idea of the previous section by bounding above the number of connected components of $\mathbb{T}^d \setminus \bigcup_{j=0}^n T_{\balpha}^{-j}(F)$, where, as previously, $F = \bigcup_{i=1}^k \partial W_i$. To achieve this goal, we use a finite union of hyperplanes whose projection contains~$F$. 
Indeed, by assumption, for each~$i$, $W_i = \bigcup_{j\in J_i} \pt(C_{i,j})$, where $J_i$ is a finite set and the~$C_{i,j}$ are convex polytopes. Let $\mathcal D= \bigcup_{i=1}^k \bigcup_{j\in J_i} C_{i,j}$. Since $\partial W_i \subset \bigcup_{j\in J_i} \pt(\partial C_{i,j})$ for each~$i$, there exists a finite set~$\mathcal{H}$ of affine hyperplanes such that $F = \bigcup_{i=1}^k \partial W_i \subset \pt(\bigcup_{H\in\mathcal{H}} H\cap \mathcal D)$.
	
Let $B(\mathbf{0},R)$ be a ball containing~$\mathcal D$; note that $\mathcal D$ is bounded as a finite union of  convex polytopes. There exists a constant $N \in \mathbb{N}$, depending only on the radius~$R$ and on the dimension~$d$, such that, for all $\ba \in \mathbb{R}^d$, there exist at most $N$ vectors $\bq \in \mathbb{Z}^d$ such that $(B(\ba,R){+}\bq) \cap [0,1]^d\neq \emptyset$. 
It follows that, for any integer~$j$ and for any hyperplane $H \in \mathcal{H}$, the set~$E_{H,j}$ of vectors $\bq \in \mathbb{Z}^d$ such that $(H \cap \mathcal D - j\balpha + \bq)\cap [0,1]^d\neq \emptyset$, has cardinality at most~$N$. By definition of the set~$E_{H,j}$, we have 
\[
T_{\balpha}^{-j}(\pt(H \cap \mathcal D)) = \pt\bigg(\bigcup_{\bq\in E_{H,j}}(H \cap \mathcal D - j\balpha + \bq) \cap [0,1]^d\bigg).
\]
It follows that, for any positive integer~$n$, 
\[
\begin{aligned}
\bigcup_{j=0}^n T_{\balpha}^{-j}(F) & \subset \pt\bigg(\bigcup_{j=0}^n \bigcup_{H\in\mathcal{H}} \bigcup_{\bq\in E_{H,j}} (H \cap \mathcal D - j\balpha + \bq) \cap [0,1]^d\bigg) \\
& \subset \pt\bigg(\bigcup_{j=0}^n \bigcup_{H\in\mathcal{H}} \bigcup_{\bq\in E_{H,j}} (H - j\balpha + \bq) \cap [0,1]^d\bigg).
\end{aligned}
\]
Let $\mathcal{C}_n$ be the set of connected components of 
\[
\Omega = (0,1)^d \setminus \bigcup_{j=0}^n \bigcup_{H\in\mathcal{H}} \bigcup_{\bq\in E_{H,j}} (H - j\balpha + \bq).
\] 
We recall that~$\mathcal{F}_n$ stands for the set of connected components of $\mathbb{T}^d \setminus \bigcup_{j=0}^n T_{\balpha}^{-j}(F)$.
On the one hand, if $V \in \mathcal{C}_n$, then $\pt(V)$ does not intersect $\bigcup_{j=0}^n T_{\balpha}^{-j}(F)$, hence $\pt(V)$ is contained in some  $U \in \mathcal{F}_n$. On the other hand, since $\pt(\Omega)$ is dense in~$\mathbb{T}^d$,  each $U \in \mathcal{F}_n$ intersects at least one~$\pt(V)$, where $V \in \mathcal{C}_n$. It follows that the map which associates with each $V \in \mathcal{C}_n$ the unique $U \in \mathcal{F}_n$ containing~$\pt(V)$  is surjective; hence 
\[
\card\mathcal{F}_n \leq \card\mathcal{C}_n.
\] 
By \eqref{e:complexity}, if $\card\mathcal{C}_n \leq Cn^d$ for some constant~$C$ that does not depend on~$n$, the proposition holds. Now, the set~$\mathcal{C}_n$ is included in the set of connected components of~$\mathbb{R}^d$ minus the union of the following sets of hyperplanes: 
\begin{itemize}
\itemsep.5ex
\item 
$H-j\balpha+\bq$ with $H\in\mathcal{H}$, $j\in\{0,\dots,n\}$, $\bq\in E_{H,j}$,
\item 
$\{(x_1,\dots,x_d) \in\mathbb{R}^d : x_i=c\}$,  $i=1,\dots, d$, $c=0,1$.
\end{itemize}
 The number of these hyperplanes is bounded above by
 \[
2d + (n{+}1) N \card\mathcal{H} \leq C_1n
 \]
for some constant~$C_1$ that does not depend on~$n$. The proposition is now a consequence of Lemma~\ref{lem:big-O}.
\end{proof}

\subsection{Lower bound for the factor complexity} \label{subsec:lower}
In Section~\ref{subsec:upper}, we obtained the upper bound $p_u(n )= O(n^d)$. 
Obtaining a lower bound is more difficult. Roughly, the idea is to find a subset~$\mathcal{G}_n$ of~$\mathcal{F}_n$, the set of connected components of $\mathbb{T}^d \setminus \bigcup_{i=0}^n T_{\balpha}^{-i}(F)$, such that, if $\bx$ and~$\by$ are in two different elements of~$\mathcal{G}_n$, then there is at least one $j \in \{1,\dots,n\}$ such that $T_{\balpha}^j(\bx)$ and~$T_{\balpha}^j(\by)$ are not in the same~$W_i$. Note that the existence statement for the linear forms $f_1, \dots, f_d$ below will be handled in Section~\ref{subsec:fi}.
	
\begin{theorem} \label{thm:complexity}
Let $T_{\balpha}$ be the translation $T_{\balpha}: \bx \in \T^d \mapsto \bx{+}\ttt \in\T^d$ assumed to be minimal.
Suppose that 
\[
\T^d = \WW_1 \cup \WW_2 \cup \dots \cup \WW_k,
\]
where the sets~$\WW_i$, $i=1,\dots,k$, are closed sets with disjoint interiors. Assume that there are $d$ independent linear forms $f_1,\dots,f_d$ on~$\mathbb{R}^d$, $d$~points $\ba_1,\dots,\ba_d \in \mathbb{R}^d$ and $r>0$ such that, for each $i \in \{1,\dots,d\}$,
\[
\pt(B(\ba_i,r) \cap \{f_i <  f_i(\ba_i)\}) \subset W_{b(i)} \quad \mbox{and} \quad \pt(B(\ba_i,r) \cap \{f_i > f_i(\ba_i)\}) \subset W_{c(i)},
\]
where $b(i) \neq c(i)$ are in $\{1,\dots,k\}$.
			
Then there exists $c>0$ such that, if $u = (u_n)_{n\in\N}$ is the coding sequence of a point $\bx_0 \in \T^d$ with $T_{\balpha}^n(\bx_0)\in \mathring{\WW}_{u_n}$ for all $n \in\N$, the factor complexity of the sequence $u=(u_n)_{n\in\N}$ satisfies  
\[
p_u(n) \geq cn^d \quad \mbox{for all}\ n \in \N.
\]
\end{theorem}

The proof of the above theorem uses the following lemma.
	
\begin{lemma} \label{lem:hyperplane}
Let $\ttt \in \mathbb{R}^d$  with $(1, \balpha)$ having rationally independent coordinates. Let $H \subset \mathbb{R}^d$ be an affine hyperplane and let $K \subset \mathbb{R}^d$ be a bounded subset. Then there exists a positive integer~$Q$, depending only on~$K$ and the direction of~$H$, such that, for all $\ba \in \mathbb{R}^d$, there are at most~$Q$ integers~$q$ such that $\ba{+}q\ttt \in K \cap H +\Z^d$.
\end{lemma}

\begin{proof}
It is enough to prove that there exists a positive integer~$Q$ depending only on~$K$ and the direction of~$H$ such that, for all $\ba \in \mathbb{R}^d$ and all $m \in \Z$, there exists at most one $k \in \Z$ satisfying $\ba {+} (m{+}kQ)\ttt \in K \cap H + \Z^d$. 
	
Suppose first that $\ba = \mathbf{0}$ and that $H$ is a vector hyperplane, i.e., $\mathbf{0} \in H$. Let $H_{\Q}$ be the vector space generated by $H \cap \Q^d$ and let $f$ be a nonzero rational linear form such that $H_{\Q} \subset \ker f$. 
Since $\ttt$ is totally irrational, we have, for all $k \in \Z \setminus \{0\}$ and all $\bq \in \Z^d$, $f(k\ttt{+}\bq) \neq 0$ and hence $k\ttt {+} \bq \notin H_{\Q}$. Furthermore, if $k\ttt{+}\bq\in H$ and if $k\ttt{+}\bq' \in H$ for some other point $\bq' \in\Z^d$, then $\bq{-}\bq' \in H \cap \Q^d \subset H_{\Q}$.
	 
Let $q_0$ be the smallest positive integer such that $q_0\ttt \in H{+}\Z^d$, if any. 
If there is no such integer $q_0>0$, then we are done because $q=0$ will be the only integer such that $q\ttt \in H{+}\Z^d$.
Let $\bq_0$ be a point in~$\Z^d$ such that $q_0\ttt{+}\bq_0 \in H$. 

Let $r>0$ be such that $K \subset B(\mathbf{0},r)$. Set $k_0 = \lceil\frac{2r}{r_0}\rceil{+}1$, where $r_0 = \dd(q_0\ttt{+}\bq_0,H_{\Q}) > 0$. Set $Q = k_0q_0$. 
We want to show that each class modulo~$Q$ contains at most one integer~$m$ such that $m\ttt \in B(\mathbf{0},r) \cap H + \Z^d$. Let $m\in \Z$ and $\bq_m\in\Z^d$ be such that $m\ttt{+}\bq_m \in B(\mathbf{0},r)\cap H$.  If $k\in\Z\setminus\{0\}$ and $\bq \in \Z^d$ are such that $(m{+}kQ)\ttt {+}\bq \in H$,  then  
\[
\bq - \bq_m - kk_0\bq_0 = (m{+}kQ)\ttt + \bq - (m\ttt{+}\bq_m) - (kk_0q_0\ttt{+}kk_0\bq_0) \in H_{\Q}.
\] 
Therefore, for any $\bu\in H_{\Q}$, we have
\[
\begin{aligned}
\|(m{+}kQ)\ttt+\bq-\bu\| & = \|(kk_0q_0\ttt{+}kk_0\bq_0{+}\bq{-}\bq_m{-}kk_0\bq_0) - \bu + (m\ttt{+}\bq_m)\| \\
& \geq \|kk_0q_0\ttt + kk_0\bq_0 + (\bq{-}\bq_m{-}kk_0\bq_0) - \bu\| - \|m\ttt{+}\bq_m\| \\
& \geq \dd(kk_0q_0\ttt{+}kk_0\bq_0, H_{\Q}) - \|m\ttt{+}\bq_m\|.
\end{aligned}
\]

Notice that, for all $t \in \mathbb{R}$ and $\ba \in \mathbb{R}^d$, $\dd(t\ba,H_{\Q}) = |t| \dd(\ba,H_{\Q})$. Therefore, 
\[
\begin{aligned}
\|(m{+}kQ)\ttt+\bq-\bu\| & \geq |kk_0| \dd(q_0\ttt{+}\bq_0,H_{\Q}) - \|m\ttt{+}\bq_m\| \\
& = |k|k_0r_0-r \geq |k| \big(\tfrac{2r}{r_0}{+}1\big) r_0 - r > r.
\end{aligned}
\]
It follows that $\dd((m{+}kQ)\ttt{+}\bq,\mathbf{0}) \geq \dd((m{+}kQ)\ttt{+}\bq,H_{\Q}) > r$, which in turn implies that $(m{+}kQ)\ttt{+}\bq \notin B(\mathbf{0},r)$.

Suppose now that $H$ is an affine hyperplane. Let $\ba \in \mathbb{R}^d$. If there is no pair $(q,\bq) \in \Z\times \Z^d$ such that $\ba{+}q\ttt{+}\bq \in K\cap H$, then we are done. Otherwise fix such a pair $(q_1,\bq_1)$. 
We are going to use the first step with the bounded subset $K' = K{-}K$ and the vector space $\overrightarrow{H} = H {-} (\ba{+}q_1\ttt{+}\bq_1)$, which is the direction of~$H$. By the first step, there exists an integer~$Q$ such that for any integer~$m'$ there is at most one integer~$k$ such that $(m'{+}kQ)\ttt \in K' \cap \overrightarrow{H} + \Z^d$. Now, for any~$m$ and~$k$, if $\ba{+}(m{+}kQ)\ttt \in K\cap H + \Z^d$, then
\[
(m{-}q_1{+}kQ)\ttt \in K\cap H - (\ba{+}q_1\ttt{+}\bq_1) + \bq_1 + \Z^d,
\]
and, since $K \cap H - (\ba{+}q_1\ttt{+}\bq_1) \subset K' \cap \overrightarrow{H}$, we have
\[
(m'{+}kQ)\ttt \in  K' \cap \overrightarrow{H}+\Z^d,
\]
with $m' = m{-}q_1$.
Therefore, given $m \in \Z$, there exists at most one integer~$k$ such that $\ba{+}(m{+}kQ)\ttt \in K \cap H + \Z^d$.
\end{proof}

\begin{proof}[Proof of Theorem~\ref{thm:complexity}]
Let $\{\be_1,\dots,\be_d\} \subset \mathbb{R}^d$ be  the dual basis of the basis $\{f_1,\dots,f_d\}$ of linear forms. Let $\|\bx\| = \max_{i=1,\dots,d} |f_i(\bx)|$ denote the supremum norm associated with the basis $\{\be_1,\dots,\be_d\}$. By reducing the positive real number~$r$ of the assumption on the~$\WW_i$, we can suppose that the balls $B(\cdot,\cdot)$ are associated with the supremum norm~$\|\cdot\|$. For $i=1,\dots,d$, set 
\[
U_i = \{\by \in \mathbb{R}^d \,:\, \|\by\| < \tfrac{r}{2},\, f_i(\by) > 0\}
\]
and, for $n \in\N$,
\[
J_{i,n} = \big\{p \in \{0,\dots,n\} \,:\, T_{\balpha}^{-p}(\pi_{\T^d}(\ba_i)) \in \pt(U_i)\big\}.
\]
Since the translation~$T_{\balpha}$ is uniquely ergodic (by assumption on~$\balpha$), one has by ergodicity $\lim_{n\to\infty} \tfrac{1}{n} \card J_{i,n} = 2c_1$, where $2c_1$ is the Lebesgue measure of~$\pi_{\T^d}(U_i)$. 
Thus there exists $n_0 \in\N$ such that, for all $n \geq n_0$ and all $i \in \{1,\dots,d\}$,
\[
\card J_{i,n} \geq c_1n
\]
Fix now $n \geq n_0$.
For each $p \in J_{i,n}$, let $\bq_{i,p} \in \Z^d$ be such that $\ba_i{-}p\ttt{+}\bq_{i,p} \in U_i$. 
Making use of Lemma~\ref{lem:hyperplane} with the hyperplanes $H_{i,p} = \{\bx \in \mathbb{R}^d : f_i(\bx) = f_i(\ba_i{-}p\ttt{+}\bq_{i,p})\}$ and the bounded sets $K_{i,p} = U_i$, we see that, for each $i \in \{1,\dots,d\}$, there exists $Q_i \in \N$ such that, for  each $p\in J_{i,n}$, there are at most~$Q_i$ integers $q \in J_{i,n}$ with $\ba_i{-}q\ttt \in H_{i,p} \cap U_i + \Z^d$. Set $Q = \max\{Q_1,\dots,Q_d\}$.
	
Observe that, if $p,q \in J_{i,n}$ are such that $f_i(\ba_i{-}p\ttt{+}\bq_{i,p}) = f_i(\ba_i{-}q\ttt{+}\bq_{i,q})$, then $\ba_i{-}q\ttt{+}\bq_{i,q}$ is both in~$U_i$ and~$H_{i,p}$, hence $\ba_i{-}q\ttt \in H_{i,p} \cap U_i + \Z^d$. 
Therefore, for each $i\in\{1,\dots,d\}$, we can find a subset $J'_{i,n} \subset J_{i,n}$,  such that
\begin{itemize}
\itemsep.5ex
\item 
$\card J'_{i,n} \geq \frac{1}{Q} c_1n = c_2n$,
\item 
$f_i(\ba_i{-}p\ttt{+}\bq_{i,p}) \neq f_i(\ba_i{-}q\ttt{+}\bq_{i,q})$ for all $p, q \in J'_{i,n}$ with $p \ne q$.
\end{itemize}
Set 
\[
U^+ = \{\by\in \mathbb{R}^d \,:\, \|\by\| < \tfrac{r}{2},\, f_i(\by) >  0,\, i=1,\dots,d\}
\]
and, for $i \in \{1,\dots,d\}$ and $p \in J'_{i,n}$, set
\[
\begin{aligned}
R_{i,p}^+ & = U^+ \cap B(\ba_i{-}p\ttt{+}\bq_{i,p},r) \cap \{\by : f_i(\by) > f_i(\ba_i{-}p\ttt{+}\bq_{i,p})\}, \\
R_{i,p}^- & = U^+ \cap B(\ba_i{-}p\ttt{+}\bq_{i,p},r) \cap \{\by : f_i(\by) < f_i(\ba_i{-}p\ttt{+}\bq_{i,p})\}.
\end{aligned}
\]
Let $p \in J'_{i,n}$. Since $\ba_i{-}p\ttt{+}\bq_{i,p} \in U_i \subset B(\mathbf{0},\tfrac12r)$, the set~$U^+$ is included in the ball $B(\ba_i{-}p\ttt{+}\bq_{i,p},r)$. Therefore,  
\[
\begin{aligned}
R_{i,p}^+ & = U^+ \cap \{\by : f_i(\by) > f_i(\ba_i{-}p\ttt{+}\bq_{i,p})\}, \\
R_{i,p}^- & = U^+ \cap \{\by : f_i(\by) < f_i(\ba_i{-}p\ttt{+}\bq_{i,p})\}.
\end{aligned}
\] 
Furthermore, if $\by\in R_{i,p}^+$, then, by definition of~$R^+_{i,p}$,  $\by = \ba_i{-}p\ttt{+}\bq_{i,p}{+}\bx$ with $\bx \in B(\mathbf{0},{\frac{r}{2}})$ and $f_i(\bx) > 0$, hence
\[
\by{+}p\ttt{-}\bq_{i,p} = \ba_i{+}\bx \in \ba_i + B(\mathbf{0},\tfrac{r}{2}) \cap \{\bz : f_i(\bz) > 0\} = B(\ba_i,\tfrac{r}{2}) \cap \{f_i > f(\ba_i)\},
\]
and, since by assumption $\pt(B(\ba_i,\frac{r}{2}) \cap \{f > f(\ba_i)\}) \subset W_{c(i)}$,
\[
T_{\balpha}^p(\pt(\by)) \in W_{c(i)}.
\]
In the same way, if $\by \in R_{i,p}^-$, then 
\[
T_{\balpha}^p(\pt(\by)) \in \WW_{b(i)}.
\]
For each $i \in \{1,\dots,d\}$, the real numbers $k_{i,p} = f_i(\ba_i{-}p\ttt{+}\bq_{i,p})$, $p \in J'_{i,n}$, are pairwise distinct and are in the interval $(0,\tfrac{r}{2})$ by definition of~$J'_{i,n}$. It follows  that the ``coordinate'' hyperplanes 
\[
\{\by \,:\, f_i(\by) = k_{i,p}\} = \bigg\{\sum_{j=1}^d y_i \be_i \in \mathbb{R}^d \,:\, y_i = k_{i,p}\bigg\},
\]
for $i \in \{1,\dots,d\}$ and $p \in J'_{i,n}$, divide~$U^+$ into a set~$E$ of nonempty coordinate parallelepipeds with cardinality
$\prod_{i=1}^d (\card J'_{i,n}{+}1) \geq (c_2n)^d$. 
Suppose that $\by$ and~$\by'$ are in two such distinct coordinate parallelepipeds. This means that there exist $i \in \{1,\dots,d\}$ and $p \in J'_{i,n}$ such that $f_i(\by) > k_{i,p} > f_i(\by')$, which implies that $\by \in R_{i,p}^+$ and $\by' \in R_{i,p}^-$, therefore $T_{\balpha}^p(\pt(\by)) \in \WW_{c(i)}$ and $T_{\balpha}^p(\pt(y') \in \WW_{b(i)}$. 
Since, for each parallelepiped $\PP \in E$, there exists a nonnegative integer~$m_{\PP}$ such that $T_{\balpha}^{m_{\PP}}(\bx_0) \in \pt(\PP) $, the factors $u_{m_{\PP}} \cdots u_{m_{\PP}+n}$, $\PP \in E$, of the sequence~$u$ encoding~$\bx_0$ are pairwise distinct. This implies that
\[
p_u(n{+}1) \geq (c_2n)^d. \qedhere
\]
\end{proof}

\subsection{End of the proof of Theorem~\ref{theo:main}} \label{subsec:fi}
We now want  to use Theorem~\ref{thm:complexity} to bound below the complexity. 
 It remains to  prove  the existence  statement  for  the linear forms $f_1,\dots , f_d$ 
from Theorem \ref{thm:complexity}.

 For each $i\in\{1,\dots,d\}$, call $W_i=\pi_{\T^{d-1}}(P_i)$ the projection of $P_i$ in the torus.
Given  $i\in\{1,\dots,d\}$,  we  now describe  the boundary of   $W_i$.
Since each~$P_i$ is a finite union of convex polytopes and since $\pi_{\T^{d-1}}^{-1}(W_i) = \bigcup_{\bx\in\Z^{d-1}} (\bx{+}P_i)$, the set $\pi_{\T^{d-1}}^{-1}(W_i)$ is locally a finite union of bounded convex polytopes.  
Given a nonzero linear form~$f$ on~$\mathbb{R}^{d-1}$ and a real number~$c$, denote $$H_{f,c} = \{\bq \in \mathbb{R}^{d-1} : f(\bq) = c\}$$ the hyperplane defined by~$f$ and~$c$.   When it is nonempty, we call the relative interior of $(-1,1)^{d-1} \cap \pi_{\T^{d-1}}^{-1}(\partial W_i) \cap H_{f,c}$ a \emph{facet} of~$W_i$ associated with~$f$ and~$c$.  

Consider the set~$W^*$ of nonzero linear forms~$f$ on~$\mathbb{R}^{d-1}$ such that  
there exist $i \in \{1,\dots,d\}$ and $c \in \mathbb{R}$ such that $(-1,1)^{d-1} \cap \pi_{\T^{d-1}}^{-1}(\partial W_i) \cap H_{f,c}$ has nonempty interior relatively to the hyperplane~$H_{f,c}$. 

Thanks to the following lemma whose proof is given below, each~$\partial W_i$ is the union of the projections of the closures of the facets of~$W_i$. Observe that each~$P_i$ is a finite union of convex polytopes with nonempty interiors because $P_i$ is the closure of its interior. 

\begin{lemma}\label{lem:boundary}
Let $K\subset \mathbb{R}^{d-1}$ be a finite union of convex polytopes with nonempty interiors. Then, for any $\bx \in\partial K$, there exist a linear form~$f$ and $c \in \mathbb{R}$ such that $H_{f,c} \cap \partial K$ has nonempty interior relative to~$H_{f,c}$ and such that $\bx$ is in the closure of this relative interior.
\end{lemma}

Moreover,  there are only finitely many hyperplanes~$H_{f,c}$ such that the intersection  $(-1,1)^{d-1} \cap \pi_{\T^{d-1}}^{-1}(\partial W_j) \cap H_{f,c}$ has nonempty interior.

Let $U$ be a  facet of some~$W_i$. Since $\pi_{\T^{d-1}}(U)$ is included in the boundary of $\bigcup_{j\neq i} W_j$, the closures of the facets of the~$W_j$, $j\neq i$, cover~$U$, and, since there are only finitely many such facets, there exist $j \neq i$ and a facet~$V$ of~$W_j$ such that $V \cap U \neq \emptyset$.  Since $W_i$ and~$W_j$ have disjoint interiors, the facets~$U$ and~$V$ must be defined by the same hyperplane.  It follows that, for each $f\in W^*$, there exist $\ba \in \mathbb{R}^{d-1}$, $i \neq j$ in $\{1,\dots,d\}$ and $r>0$ such that
\[
\pi_{\T^{d-1}}(B(\ba,r) \cap \{f < f(\ba)\}) \subset W_i \quad \mbox{and} \quad \pi_{\T^{d-1}}(B(\ba,r) \cap \{f > f(\ba)\}) \subset W_j.
\]

Thus, to use Theorem~\ref{thm:complexity}, we only have to prove that the linear forms~$f$ in~$W^*$ generate the vector space of all linear forms on~$\mathbb{R}^{d-1}$. Suppose on the contrary that $W^*$ does not generate the vector space of all linear forms. With this assumption, the intersection $\bigcap_{f\in W^*} \ker f$ contains a nonzero vector~$\mathbf{v}$. Consider the set~$\mathcal{H}$ of all hyperplanes~$H_{f,c}$ associated with a facet of one of the~$W_j$, and let $G = \bigcup_{H\in\mathcal{H}} H {+} \mathbb{Z}^{d-1} {+} \mathbb{Z}\balpha$. Since $G$ is a countable union of hyperplanes, $\mathbb{R}^{d-1} \setminus G$ is nonempty. 
Moreover by definition of~$G$, if $\ba \notin G$, then $(\ba{+}\mathbb{R} \mathbf{v}) \cap G =\emptyset$, therefore $\pi_{\T^{d-1}}(\ba{+}\mathbb{R} \mathbf{v})$ does not meet~$\pi_{\T^{d-1}}(G)$, which in turn implies that $\pi_{\T^{d-1}}(\ba{+}\mathbb{R} \mathbf{v})$ does not meet $\bigcup_{n\in\mathbb{Z}} T_{\balpha}^n(\bigcup_{j=1}^d \partial W_j)$.  
It follows that, if $\bx \in\mathbb{T}^{d-1} \setminus \pi_{\T^{d-1}}(G)$, then all the points in $\bx{+}\mathbb{R}\mathbf{v}$ have the same coding sequence, which contradicts Remark~\ref{r:differentcoding}.

\begin{proof}[Proof of Lemma \ref{lem:boundary}] Let $K=K_1\cup\dots\cup K_n\subset\mathbb R^{d-1}$ be finite union of convex polytopes and let $\bx\in\partial K$. Let  $\mathcal F$ be the set of   closures of the facets of the convex polytopes $K_1,\dots, K_n$, and let $\mathcal F_{\bx}$ be the set of $F\in\mathcal F$ that contains $\bx$. The  facets in $\mathcal F_{\bx}$  are defined by finitely many hyperplanes, $H_1,\dots,H_N$. Let $r=\tfrac12\dd(\bx,\cup F)$ where the union is over all the facets $F\in\mathcal F\setminus \mathcal F_{\bx}$.  The set 
	\[
	B(\bx,r)\setminus(H_1\cup\dots\cup H_N)
	\]
	is finite union of open truncated polyhedral cones of apex $\bx$. Each of these cones is contained either in $\mathbb R^{d-1}\setminus K$, or in $K$, because these cones intersect none of the boundaries of the $K_i$. It follows that one cone  contained in $K$ and one cone  contained in $\mathbb R^{d-1}\setminus K$ share a common facet. This facet is contained in a facet of~$K$ and its closure contains~$\bx$. 
\end{proof}

\section{Additional comments}\label{sec:comments}
By definition, Tijdeman sequences with parameters $C=C'=1{-}\frac{1}{2d-2}$ are fairly distributed, i.e., $\Delta_{\balpha}(u) \leq 1{-}\frac{1}{2d-2}$.
In the case $d=2$, fairly distributed sequences and Tijdeman sequences coincide, and they are Sturmian sequences; see Remark~\ref{rem:sturm}. 
Moreover we have seen in Remark~\ref{rem:billiardevenly} that some hypercubic billiard sequences are also fairly distributed (and they are also Tijdeman sequences by Remark~\ref{rem:billard}). 
However, we do not expect all fairly distributed sequences to be Tijdeman sequences.

The present study raises the following natural questions.
\begin{itemize}
\item
What happens when we consider not only a notion of discrepancy based on occurrences of letters but based on occurrences of factors? 
\item
What happens when the  frequency vector~$\balpha$ has rationally dependent coordinates, and even when all the entries of~$\balpha$ are rational? 
\item
Are there finite sequences with discrepancy smaller than or equal to $1 {-} \frac{1}{2d-2}$ that cannot  be prolonged  into a fairly distributed sequence? 
\end{itemize}

We turn to questions related to the  factor complexity function~$p_u(n)$. If $u$ is a symbolic coding of a piecewise translation map associated with a minimal translation on the torus~$\mathbb{T}^{d-1}$, then it  is shown in \cite{BB:2013} that $p_u(n)\geq (d{-}1)n {+} 1$ for each~$n$. There  exist   sequences~$u$ having a bounded discrepancy function~$\Delta_{\balpha}(u)$ and 
having also a factor complexity of smaller order than Tjideman sequences; they even have linear factor complexity whereas Tijdeman sequences have factor complexity of order~$d{-}1$. However, the price to pay when reducing factor complexity seems to yield an increase of the discrepancy. More precisely, a construction of sequences     having both finite discrepancy and linear factor complexity in dimension $d=3$ for a.e.~$\balpha$ is provided in \cite{BST:20}  with constructions based on the  Cassaigne--Selmer multidimensional continued fraction algorithm. These symbolic codings thus enjoy the striking  properties of Sturmian sequences combining linear factor complexity and good local discrepancy properties. 
But the  corresponding fundamental domains have  fractal boundary. They are obtained as so-called Rauzy fractals which are known to  provide suitable and effective   windows for cut and  project schemes as well as  fundamental domains for toral translations.  We end with the following questions.

\begin{itemize}
\item
When $d\geq 3$, what is the lowest  bound for $\Delta_{\balpha}(u)$ when we restrict  to sequences~$u$ with linear factor complexity? 
\item
Does the following hold for~$\balpha$ a totally irrational frequency vector:
If $u$ is such that $\Delta_{\balpha} (u) \leq D_d$, then there exists $C_u >0 $ such  that $p_u(n) \geq C_u n^{d-1}$ for all~$n$ (where $d$ is the size of the alphabet). 
\end{itemize}

\bibliographystyle{amsalpha}
\bibliography{tijdeman}
\end{document}